\documentclass[reqno,oneside]{amsart}

\usepackage[T1]{fontenc}
\usepackage[utf8]{inputenc}
\usepackage[english]{babel}
\usepackage{microtype,indentfirst,amsmath,amssymb,mathtools,braket,amsthm,booktabs,caption,graphicx,tabularx,siunitx,array,rotating,sidecap,subfig,wrapfig,afterpage,blindtext,color,mathrsfs,tikz-cd,inputenc,cancel,pgfplots,tikz,verbatim,enumitem,geometry}

\usepackage[autostyle]{csquotes}
\usepackage[backend=biber,style=numeric]{biblatex}
\usepackage{guit}

\usepackage[final]{pdfpages}

\usepackage[english]{varioref}
\usepackage{hyperref}
\hypersetup{hidelinks}

\newcommand{\numberset}{\mathbb}
\newcommand{\N}{\numberset{N}}

\newcommand{\R}{\numberset{R}}
\newcommand{\C}{\numberset{C}}

\renewcommand{\epsilon}{\varepsilon}
\renewcommand{\phi}{\varphi}

\DeclarePairedDelimiter{\norm}{\lVert}{\rVert}

\DeclareMathOperator{\Realpart}{Re}
\renewcommand{\Re}{\Realpart}
\DeclareMathOperator{\Impart}{Im}
\renewcommand{\Im}{\Impart}

\DeclareMathOperator{\Mat}{Mat}
\DeclareMathOperator{\codim}{codim}

\DeclareMathOperator{\dist}{dist}

\newcommand{\ddc}{\mathrm{dd^c}}
\renewcommand{\d}{\mathrm{d}}
\newcommand{\dc}{\mathrm{d^c}}

\renewcommand{\vec}{\boldsymbol}

\theoremstyle{definition}
\newtheorem{definition}{Definition}[section]

\newtheorem{introdef}{Definition}

\theoremstyle{plain}
\newtheorem{theorem}[definition]{Theorem}
\newtheorem{lemma}[definition]{Lemma}
\newtheorem{corollary}[definition]{Corollary}
\newtheorem{proposition}[definition]{Proposition}

\newtheorem*{mainthm}{Main Theorem}

\theoremstyle{remark}
\newtheorem{remark}[definition]{Remark}

\newtheorem{intrormk}[introdef]{Remark}

\pgfplotsset{compat=1.18}

\makeatletter
	\renewcommand{\paragraph}{
		\@startsection{paragraph}{4}
		{\z@}
		{2ex \@plus 0.5ex \@minus 0.2ex}
		{-1em}
		{\normalfont\bfseries}
	}
	\renewcommand{\subsection}{
		\@startsection{subsection}{2}
		{\z@}
		{.7\linespacing \@plus\linespacing}
		{.5\linespacing}
		{\normalfont\scshape\centering}
}
\makeatother

\title{Analytic structure of $q$-pseudoconcave subsets of continuous graphs}
\author{Filippo Valnegri}
\address{Department of Mathematics and Applications, University of Milano-Bicocca, Milano, 20126, Italy}
\email{\href{mailto:filippo.valnegri@unimib.it}{filippo.valnegri@unimib.it}}
\date{}

\begin{document}
	
\begin{abstract}
	The aim of this article is to find under which conditions there exists a foliation by $n$-dimensional complex manifolds on a closed subset $Z\subset\C^N$, which is locally a continuous graph over a closed subset of $\C^n\times\R$. We prove that such foliation does exist if $Z$ is $q$-pseudoconcave, for any $q\geq n$. We also prove that this bound on the index of pseudoconcavity is sharp. Namely, we construct many $q$-pseudoconcave subsets of $\C^N$, for $q<n$, which are smooth graphs over closed subsets of $\C^n\times\R$ but with no analytic structure. As an application, we show that an $n$-pseudoconcave set sitting in a subset of $\C^N$, which is locally a differentiable graph over a domain of $\R^{2n+1}$, is foliated by $n$-dimensional complex manifolds.
\end{abstract}


\maketitle

\section*{Introduction}
The existence of an analytic structure inside a subset of $\C^N$ (for $N\geq 2$) is an interesting and fundamental problem in the field of several complex variables. Namely, given a subset $Z\subset\C^N$ and a fixed point $p\in Z$, when does there exist an analytic variety of positive dimension passing through $p$ and contained in $Z$?

Clearly, such an analytic variety does not exist in general. In particular, for each $N\geq 2$ we can find a closed subset $E\subset\C^N$ which contains no analytic varieties of positive dimension (see~\cite[Theorem~1.1]{har_shch_tom:wermer}). On the other hand, various results lead to a positive answer for subsets which can be represented as graphs. For this class of subsets, the sufficient condition required depends on the impossibility to extend holomorphic functions defined on the complement of $Z$. This impossibility is clearly a necessary condition for the existence of an analytic variety in $Z$. Indeed, if there exists an analytic variety $A$ such that $p\in A\subset Z$, then $A$ is the zero locus of a holomorphic function and thus we can find a holomorphic function on $\C^N\smallsetminus Z$ which does not extend to $A$. Conversely, Hartogs~\cite{hartogs:analytischen} proved that, if $Z\subset\C^N$ is a continuous graph over a domain $D\subset\C^{N-1}$ such that its complement $(D\times\C)\smallsetminus Z$ is a domain of holomorphy, then $Z$ is a complex hypersurface. Furthermore, Tréprau~\cite{trepreau:prolongement} showed that the same conclusion holds for reasonably smooth real hypersurfaces. Namely, if $Z\subset\C^N$ is a $\mathscr{C}^2$ graph over a domain $D\subset\C^{N-1}\times\R$ and its complement $(D\times i\R)\smallsetminus Z$ is the disjoint union of two domains of holomorphy, then $Z$ is union of complex hypersurfaces. Trépreau's Theorem was generalised to the case of continuous graphs by Shcherbina~\cite{shcherbina:polynomial_hull} (for $N=2$) and Chirka~\cite{chirka:trepreau_theorem} (for $N\geq 2$).

By Levi's Theorem, the assumption on the complement of $Z$ can be rephrased in terms of a pseudoconcavity property of $Z$ itself. Adopting this perspective, Pawlaschyk and Shcherbina~\cite{pawl_shcher:foliations} extended Chirka's Theorem to higher codimension graphs. Namely, if $Z\subset\C^N$ is a continuous graph over a domain $D\subset\C^n\times\R$ (for $n<N$) and it is an $n$-pseudoconcave subset of $D\times i\R\times\C^p$ (for $p=N-n-1\geq 0$), then it is foliated by $n$-dimensional complex manifolds.
\begin{intrormk}
	Here and in what follows we use the term \emph{$q$-pseudoconcave} in the sense of Słodkowski (see Definition~\ref{def:pseudoconcave}), and the term \emph{foliated} in a generalised sense (see Definition~\ref{def:foliated}).
\end{intrormk}

By considering $q$-pseudoconcavity sets instead of domains of holomorphy, we are drastically changing the nature of the problem. Indeed, while domains of holomorphy are governed by holomorphic functions, $q$-pseudoconcave sets are governed by $q$-plurisubharmonic ones. This new perspective suggests that the existence of an analytic structure on $Z$ is intrinsically governed by the behaviour of $q$-plurisubharmonic functions on it. 

This realisation naturally leads us to consider another equivalent, yet far more flexible, property: the $q$-local maximum property (see Definition~\ref{def:locmaxset}). Indeed, this property tells us exactly how $q$-plurisubharmonic functions behave locally on a set. The equivalence between $q$-pseudoconcavity and the $q$-local maximum property was established by Słodkowski, who proved, for $V\subset\C^N$ open, that $Z\subset V$ is a $q$-pseudoconcave subset of $V$ if and only if it is a $(q-1)$-local maximum set (see~\cite[Theorem~4.2]{slodkowski:loc_max_property}).

The $q$-local maximum property is a very powerful tool, which is already crucial in various studies of different nature. For example, it is a key factor in the study of obstructions to the existence of strictly plurisubharmonic exhaustion functions on complex manifolds (see~\cite{mon_slod_tom:weaklycomplete,mon_tom:compact,slod_tom:minimal_kernels}), but it also arises in many other problems (e.g., see~\cite{pol_shch:separable,slodkowski:decomposition}).

In this paper, using the full strength of the $q$-local maximum property, we prove the existence of an analytic structure on a class of subsets of $\C^N$ much wider than the one previously known. To this end, we adapt and extend some of the techniques already implemented by Pawlaschyk and Shcherbina, in order to operate directly with the $q$-local maximum property, rather than with $q$-pseudoconcavity. Namely, we prove the following Theorem.
\begin{mainthm}
	Let $Z\subset\C^N$ be an $n$-pseudoconcave subset (for $1\leq n<N$) which is locally the graph of a continuous function $g:X\to\R\times\C^p$ (for $p=N-n-1\geq 0$), where $X\subset\C^n\times\R$ is a closed subset. Then $Z$ is foliated by $n$-dimensional complex manifolds.
\end{mainthm}
Since a $q$-pseudoconcave set is also $m$-pseudoconcave for any $m\leq q$, the Main Theorem also shows that the same conclusion holds for every $q$-pseudoconcave subset of $\C^N$, which is locally a continuous graph over a closed subset of $\C^n\times\R$, for any $q\geq n$.

It is then natural to ask if this bound on the index of pseudoconcavity is sharp. By properly modifying the closed subset $E\subset\C^N$ with no analytic structure given by~\cite[Theorem~1.1]{har_shch_tom:wermer}, we indeed show the sharpness of this bound. Namely, for each $q<n$ we construct many $q$-pseudoconcave subsets of $\C^N$, which are globally smooth graphs over closed subsets of $\C^n\times\R$ but with no analytic structure.

\paragraph{Organisation of the article.}
In Section~\ref{sec:first_gen} we extend Chirka's Theorem to $(N-1)$-pseudoconcave subsets of continuous graphs over domains of $\C^{N-1}\times\R$. Moreover, we obtain foliations on $(N-1)$-pseudoconcave sets which are continuous graphs themselves, over closed subsets of $\C^{N-1}\times\R$.

In Section~\ref{sec:locmax_set} we gather some properties about $q$-local maximum sets and the existence of certain $q$-plurisubharmonic functions when a continuous graph does not have the $q$-local maximum property.

In Section~\ref{sec:limit} we prove that the upper closed limit of a net of $q$-local maximum sets is still a $q$-local maximum set.

In Section~\ref{sec:second_gen} we prove our Main Theorem (see Theorem~\ref{thm:foliation'}). This is done constructing a foliation on an underlying set to the one considered and then lifting this foliation using the function which defines the graph.

In Section~\ref{sec:applications} we gather some applications of our Main Theorem. For example, we obtain foliations on $n$-pseudoconcave sets contained in subsets of $\C^N$ which are locally differentiable graphs over domains of $\R^{2n+1}$.

In Section~\ref{sec:sharpness} we show that the choice for index of pseudoconcavity in the Main Theorem is sharp. Namely, we construct in $\C^N$ an $(n-1)$-pseudoconcave subset of a smooth graph, defined over $\C^n\times\R$, which has no analytic structure.

\paragraph{Acknowledgments.} The author would like to thank Zbigniew Słodkowski for the helpful suggestions and remarks on a preliminary version of this paper and his PhD supervisor, Samuele Mongodi, for his guidance and the useful discussions about this work.

\section{The hypersurface case}
\label{sec:first_gen}
In this section, we prove that every $n$-pseudoconcave subset of a continuous graph, which is a real hypersurface in $\C^{n+1}$, is foliated by complex hypersurfaces.

First of all, let us recall the definition of $q$-pseudoconvexity and $q$-pseudo\-concavity we are referring to.
\begin{definition}
\label{def:pseudoconvex}
	Let $U,V\subset\C^n$ be two open subsets, with $U\subset V$. We say that $U$ is \emph{$q$-pseudoconvex in $V$} (for $0\leq q\leq n-1$) if there exists an open neighbourhood $W\subset V$ of $V\cap bU$ such that the function
	\[
		\phi:W\cap U\to\R, \quad \phi(z)\coloneqq-\log\bigl(\dist(z,bU)\bigr)
	\]
	is $q$-plurisubharmonic, where $bU$ denotes the topological boundary of $U$.
\end{definition}
\begin{remark}
	There are various definitions of pseudoconvexity. The one stated in Definition~\ref{def:pseudoconvex} is due to Słodkowski. However, in the smooth setting the following are equivalent:
	\begin{enumerate}
		\item[(a)] $U$ is $q$-pseudoconvex in $V$ in the sense of Słodkowski;
		\item[(b)] $U$ is Hartogs $(n-q-1)$-pseudoconvex relatively to $V$;
		\item[(c)] $U$ is Levi $q$-pseudoconvex in $V$;
		\item[(d)] $U$ is $(q+1)$-convex in the sense of Grauert--Andreotti.
	\end{enumerate}
\end{remark}
\begin{definition}
\label{def:pseudoconcave}
	Let $V\subset\C^n$ be an open subset and let $Z\subset V$ be a closed one. We say that $Z$ is a \emph{$q$-pseudoconcave subset} of $V$ (for $1\leq q\leq n-1$) if $U\coloneqq V\smallsetminus Z$ is $(n-q-1)$-pseudoconvex in $V$.
\end{definition}

As anticipated in the introduction, our aim is to construct a foliation on a closed subset of $\C^n$. Since a closed subset does not have a priori any differentiable structure (e.g., local charts), we cannot construct on it a foliation in the classical sense. We will then use the term \emph{foliated} is the following sense (the name is motivated by analogy with the $\mathscr{C}^2$ case).
\begin{definition}
\label{def:foliated}
	Let $Z\subset\C^n$ be a closed subset. We say that $Z$ is \emph{foliated} if it can be written as the disjoint union of some embedded submanifolds, all of the same dimension.
\end{definition}
\begin{remark}
	Note that we do not assume any kind of transversal geometry for such foliations.
\end{remark}

After these preliminary definitions, we can now delve into the main purpose of this section. Let us formally state its main result. 
\begin{theorem}
\label{thm:foliation}
	For $n\geq1$, let $D\subset\C^n\times\R$ be a domain. Let $g:D\to\R$ be a continuous function and let
	\[
		\Gamma(g)=\bigl\{(z,w)\in D\times i\R \mid \Im(w)=g\bigl(z,\Re(w)\bigr)\bigr\}
	\]
	be its graph in $\C^{n+1}=\C^n_z\times(\R\times i\R)_w$. If $Z_0\subset\Gamma(g)$ is an $n$-pseudoconcave subset, then there exists a unique foliation on it given by $n$-dimensional complex manifolds.
\end{theorem}
\begin{remark}
	When we say that $Z_0\subset\Gamma(g)$ is an $n$-pseudoconcave subset, we mean that $Z_0$ is a subset of $\Gamma(g)$ which is $n$-pseudoconcave in $D\times i\R$.
\end{remark}

In order to obtain a foliation on $Z_0$, we will consider a foliated set in $\Gamma(g)$ (constructed by Chirka), prove that it contains $Z_0$ and then show that $Z_0$ is the union of some leaves of this foliation.

First of all, let us recall a classical construction (see~\cite[Lemma~1.2]{trepreau:prolongement}). Let $D\subset\C^n\times\R$ be a domain and let $g:D\to\R$ be a continuous function. Then there exist functions $u,v:D\to\R$ upper and lower semicontinuous (respectively), with $u\leq g\leq v$, such that the domains 
\begin{align*}
	\Gamma_+(u) &= \bigl\{(z,w)\in D\times i\R \mid \Im(w)>u\bigl(z,\Re(w)\bigr)\bigr\} \\
	\Gamma_-(v) &= \bigl\{(z,w)\in D\times i\R \mid \Im(w)<v\bigl(z,\Re(w)\bigr)\bigr\}
\end{align*}
are pseudoconvex. Namely, $u$ is the infimum of all the upper semicontinuous functions $\tilde{u}\leq g$ on $D$ such that every holomorphic function on $\Gamma_+(g)$ extends holomorphically to $\Gamma_+(\tilde{u})$ and $v$ is the supremum of all the lower semicontinuous functions $\tilde{v}\geq g$ on $D$ such that every holomorphic function on $\Gamma_-(g)$ extends holomorphically to $\Gamma_-(\tilde{v})$.

\begin{remark}
\label{rem:holenvelope}
	By construction, $\Gamma_+(u)$ is the envelope of holomorphy of $\Gamma_+(g)$ and $\Gamma_-(v)$ is the envelope of holomorphy of $\Gamma_-(g)$.
\end{remark}

Let
\[
	E\coloneqq\Gamma(u)\cap\Gamma(v)\subset\Gamma(g)
\]
By~\cite[Corollary to Theorem~2]{chirka:trepreau_theorem}, $E$ is foliated by $n$-dimensional complex manifolds.
\begin{remark}
	The set $E$ was first introduced by Trépreau in the $\mathscr{C}^2$ setting (see~\cite[Theorem~1.1]{trepreau:prolongement}).
\end{remark}

\begin{lemma}
\label{lemma:contained}
	Let $Z_0\subset\Gamma(g)$ be an $n$-pseudoconcave subset. Then $Z_0$ is contained in $E$. 
\end{lemma}
\begin{proof}
	To prove this, we need to show that $Z_0$ is contained in both the graphs of $u$ and $v$. For this purpose, we will show that $Z_0$ is disjoint from both the epigraph and the hypograph of these functions.
	
	First of all, let us prove that $Z_0$ is disjoint from $\Gamma_-(u)$. Since $Z_0\subset\Gamma(g)$, we have that $Z_0\cap\Gamma_-(g)=\emptyset$. On the other hand, $\Gamma_-(u)\subset\Gamma_-(g)$, since $u\leq g$. Hence $Z_0$ is disjoint from $\Gamma_-(u)$.
	
	Now let us prove that $Z_0$ is disjoint from $\Gamma_+(u)$. By contradiction, suppose that $Z_0\cap\Gamma_+(u)\neq\emptyset$. Let $W$ be the connected component of $\Gamma_+(u)\smallsetminus Z_0$ containing $\Gamma_+(g)$ and let $p\in Z_0\cap\Gamma_+(u)$ be a point belonging to the boundary of $W$. Since $Z_0$ is $n$-pseudoconcave, then $\Gamma_+(u)\smallsetminus Z_0$ is pseudoconvex and thus $W$ is a domain of holomorphy. Hence there exists a holomorphic function $F:W\to\C$ which cannot be extended to any neighbourhood of $p$. However, $F|_{\Gamma_+(g)}$ can be extended to some holomorphic function $H:\Gamma_+(u)\to\C$, by Remark~\ref{rem:holenvelope}. Since $\Gamma_+(g)\subset W$ is open, then $F=H|_W$, i.e., $H$ extends $F$ to a neighbourhood of $p$. This is a contradiction. Therefore $Z_0$ is disjoint from $\Gamma_+(u)$.
	
	Therefore $Z_0$ is contained in $\Gamma(u)$. Analogously, $Z_0$ is contained in $\Gamma(v)$. Thus
	\[
		Z_0\subset\Gamma(u)\cap\Gamma(v)=E \qedhere
	\]
\end{proof}

\begin{lemma}
	Let $\mathcal{F}$ be the foliation of $E$ by complex hypersurfaces. Then $Z_0$ is union of some leaves of $\mathcal{F}$.
\end{lemma}
\begin{proof}
	Let $\ell$ be a leaf of $\mathcal{F}$ which intersects $Z_0$ and let $(z_0,w_0)\in\ell\cap Z_0$. It is sufficient to show that there exists a neighbourhood of $(z_0,w_0)$ in $\ell$ which is all contained in $Z_0$. Note that, by~\cite[Theorem~2]{chirka:trepreau_theorem}, 
	\begin{equation}
	\label{eq:graph_form}
		\ell=\Gamma(f)=\{(z,f(z))\in\C^{n+1} \mid z\in G\}
	\end{equation}
	where $f:G\to\C$ is a holomorphic function and $G\subset\C^n$ is a contractible domain of holomorphy. Since $(z_0,w_0)\in\ell$, then $z_0\in G$ and $f(z_0)=w_0$. On the other hand, since $Z_0\subset V\coloneqq D\times i\R$ is $n$-pseudoconcave, then $U\coloneqq V\smallsetminus Z_0$ is pseudoconvex.

	Let $\lambda\subset\C^n$ be a 1-dimensional affine complex subspace passing through $z_0$. Then $(z_0,w_0)\in(\lambda\times\C)\cap E$. Let
	\[
		W=\{z\in\C^n \mid |z-z_0|<r\}
	\]
	be an open ball centred at $z_0$ of radius $r>0$, such that $\overline{W}\subset G$. Let $W_\lambda\coloneqq W\cap\lambda$, which is a 1-dimensional disc centred at $z_0$, and set $\Delta_\lambda\coloneqq\Gamma(f|_{W_\lambda})$. Then $\Delta_\lambda$ is an analytic disc in $\C^{n+1}$. By~\eqref{eq:graph_form}, we have that
	\[
		\ell\supset\bigcup_{\lambda\ni z_0} \Delta_\lambda=\Gamma(f|_W)\eqqcolon \Delta
	\]
	Note that $\Delta$ is a neighbourhood of $(z_0,w_0)$ in $\ell$. Now, for $s\in(0,S]$ let
	\[
		\Delta_\lambda^s\coloneqq\Delta_\lambda+(\vec{0},is)=\{(z,f(z)+is)\in\C^{n+1} \mid z\in W_\lambda\}
	\]
	Since $\overline{\Delta_\lambda^s}\subset\Gamma(g+s)$ and $Z_0\subset\Gamma(g)$, then $\overline{\Delta_\lambda^s}\cap Z_0=\emptyset$ and thus $\overline{\Delta_\lambda^s}\subset U$. Note that $\overline{\Delta_\lambda^s}\to\overline{\Delta_\lambda}$ and $b\Delta_\lambda^s\to b\Delta_\lambda$ for $s\to0^+$, in the Hausdorff metric. If $\overline{\Delta_\lambda}\cap U\neq\emptyset$, then $\overline{\Delta_\lambda}\subset U$, by the ``strong continuity principle'' (see~\cite[Section~17.3]{vladimirov:methods}). This is a contradiction, since $(z_0,w_0)\in \Delta_\lambda$. Thus $\overline{\Delta_\lambda}\subset Z_0$.

	By the arbitrariness of $\lambda$, we obtain that $\Delta\subset Z_0$. Note that $\Delta$ is a neighbourhood of $(z_0,w_0)$ and thus $\ell\cap Z_0$ is open in $\ell$. On the other hand, $\ell\cap Z_0$ is clearly closed in $\ell$, since $Z_0$ is closed. Hence $\ell\subset Z_0$, since $\ell$ is connected. Therefore $Z_0$ is union of leaves of $\mathcal{F}$.
\end{proof}

\begin{remark}
	The foliation obtained on $Z_0$ is unique. Namely, consider the foliation $\{\ell_\alpha\}$ constructed above and suppose that there exists another foliation $\{\ell'_\beta\}$ by $n$-dimensional complex manifolds on $Z_0$. Note that a leaf $\ell_\alpha$ is a connected $n$-dimensional complex analytic set in $Z_0\subset E\subset\Gamma(g)$, which is closed in $D\times i\R$, by~\cite[Theorem~2]{chirka:trepreau_theorem}. Take a leaf $\ell'_\beta$ of the second foliation such that $\ell_\alpha\cap\ell'_\beta\neq\emptyset$. Since $\ell'_\beta$ is a connected complex analytic set, then $\ell'_\beta\subset\ell_\alpha$, by~\cite[Lemma~2]{chirka:trepreau_theorem}. Therefore we have the uniqueness of the foliation.
\end{remark}

We have thus proved Theorem~\ref{thm:foliation}. 

Furthermore, we can see $Z_0$ not only as a closed subset of a graph, but as a graph itself.
\begin{corollary}
\label{coro:foliation}
	For $n\geq1$, let $X_0\subset\C^n\times\R$ be a closed subset. Let $g_0:X_0\to\R$ be a continuous function and let $Z_0\coloneqq\Gamma(g_0)$ be its graph. If $Z_0$ is an $n$-pseudoconcave subset of $\C^{n+1}$, then there exists a unique foliation on it given by $n$-dimensional complex manifolds.
\end{corollary}
\begin{proof}
	By Tietze Extension Theorem, there exists a continuous function $g:\C^n\times\R\to\R$ which extends $g_0$. Then $Z_0$ is an $n$-pseudoconcave subset of $\Gamma(g)$. Therefore, by Theorem~\ref{thm:foliation}, we conclude.
\end{proof}

Finally, since the result we have proven is local, we can also ask $Z_0$ to be just locally a continuous graph, instead of globally.
\begin{corollary}
\label{coro:foliation2}
	Let $Z\subset\C^{n+1}$ be an $n$-pseudoconcave subset, which is locally the graph of a continuous function over a closed subset of $\C^n\times\R$. Then there exists a unique foliation on $Z$ given by complex hypersurfaces.
\end{corollary}

\section{Some properties of \texorpdfstring{$q$}{q}-local maximum sets}
\label{sec:locmax_set}
In this section, we prove some properties regarding $q$-local maximum sets, which will be fundamental in Section~\ref{sec:second_gen}.

First of all, let us recall their definition.
\begin{definition}
\label{def:locmaxset}
	Let $Z\subset\C^n$ be a locally closed subset. We say that $Z$ is a \emph{$q$-local maximum set} (for $0\leq q\leq n-1$) if for every $z\in Z$ there exists an open neighbourhood $U\subset\C^n$ of $z$ such that for every compact subset $K\subset U$, with $Z\cap K\neq\emptyset$ compact, and for every $q$-plurisubharmonic function $\phi:U\to\R$ we have that
	\[
		\max_{Z\cap K}\phi=\max_{Z\cap bK}\phi
	\]
	where $bK$ denotes the topological boundary of $K$.
\end{definition}

In~\cite[Theorem~5.1]{slodkowski:loc_max_property}, a characterisation of $q$-local maximum sets is given. We are interested, in particular, in property~(iii): $Z\subset\C^n$ is a $q$-local maximum set if and only if it does not exist a $q$-plurisubharmonic function $\rho:Z\to\R$ such that $\rho(z_0)=0$ and $\rho<0$ on $Z\smallsetminus\{z_0\}$, for some $z_0\in Z$. Note that the function $\rho$ is, a priori, only upper semicontinuous. We show that, at least in a neighbourhood of a point, this function can be taken continuous.
\begin{lemma}
\label{lemma:lms_basis'}
	Let $Z\subset\C^n$ be a closed subset. If $Z$ is not a $q$-local maximum set, then there exists a point $z_0\in Z$ such that: for every neighbourhood $V\subset\C^n$ of $z_0$ there exists a continuous $q$-plurisubharmonic function $\rho:U\to\R$, where $U\subset\C^n$ is a neighbourhood of $Z\cap V$, and a point $w_0\in Z\cap V$ such that
	\begin{gather*}
		\rho(w_0)=0 \\
		\rho(w)<0 \quad \forall w\in (Z\cap V)\smallsetminus\{w_0\}
	\end{gather*}
\end{lemma}
\begin{proof}
	Since $Z$ has not the $q$-local maximum property, then there exists a point $z_0\in Z$ such that for each neighbourhood $V\subset\C^n$ of $z_0$ there exists a compact subset $C\subset V$ and a smooth $q$-plurisubharmonic function $\phi:V\to\R$ such that
	\[
		\max_{Z\cap C}\phi>\max_{Z\cap bC}\phi
	\]
	Fix $V$ and take the respective $C$ and $\phi$.
	
	We construct the continuous function $\rho$ using the same strategy of~\cite[Theorem~5.1]{slodkowski:loc_max_property}. By~\cite[Lemma~2.2]{slodkowski:loc_max_property}, there exists a point $w_0\in Z\cap\mathring{C}$ and a smooth strictly plurisubharmonic function $f:\C^n\to\R$ such that
	\begin{gather*}
		(\phi+f)(w_0)=0 \\
		(\phi+f)(w)<0 \quad \forall w\in(Z\cap C)\smallsetminus\{w_0\}
	\end{gather*}
	Take $d<0$ such that
	\[
		\max_{Z\cap bC}(\phi+f)<d
	\]
	and let $H=\{z\in V \mid (\phi+f)(z)<d\}$. Then $H\subset V$ is open and $Z\cap bC\subset H$. Let $U=\mathring{C}\cup H\cup(V\smallsetminus C)$, then $U$ is an open neighbourhood of $Z\cap V$. Let us define
	\[
		\rho:U\to\R, \quad \rho(z)=
		\begin{cases}
			\max\{(\phi+f)(z),d\}, & z\in\mathring{C} \\
			d, & z\in H\cup(V\smallsetminus C)
		\end{cases}
	\]
	Since $\phi+f<d$ on $H$, then $\rho$ is well-defined. By construction, we have that $\rho$ is a continuous $q$-plurisubharmonic function and $w_0$ is a strict global maximum point for $\rho|_{Z\cap V}$.
\end{proof}

Now we prove a topological property of continuous functions.
\begin{lemma}
\label{lemma:bases}
	Let $X$ and $Y$ be topological spaces and let $g:X\to Y$ be a continuous function. Let $\mathcal{B}_X$ be a basis for the topology of $X$ and $\mathcal{B}_Y$ be a basis for the topology of $Y$. Then for every open subset $V\subset X\times Y$, which intersects the graph $\Gamma(g)$, and for every point $(x_0,g(x_0))\in V\cap\Gamma(g)$ there exist two basic open sets $B\in\mathcal{B}_X$ and $D\in\mathcal{B}_Y$ such that $(x_0,g(x_0))\in  B\times D\subset V$ and $g(\overline{B})\subset D$.
\end{lemma}
\begin{proof}
	Recall that $\mathcal{B}_X\times\mathcal{B}_Y$ is a basis for the topology of $X\times Y$. Hence, fixed $(x_0,g(x_0))\in V\cap\Gamma(g)$, there exists a basic open $B\times D\in\mathcal{B}_X\times\mathcal{B}_Y$ such that
	\[
		(x_0,g(x_0))\in B\times D\subset V
	\]
	Note that $B\cap g^{-1}(D)$ is an open neighbourhood of $x_0$ in $X$. Hence there exists a basic open set $B_0\in\mathcal{B}_X$ such that $x_0\in\overline{B_0}\subset B\cap g^{-1}(D)$. In particular, we have that
	\[
		g(\overline{B_0})\subset g(B\cap g^{-1}(D))\subset g(g^{-1}(D))\subset D
	\]
	Therefore, substituting $B$ with $B_0$, we conclude.
\end{proof}

Combining Lemma~\ref{lemma:lms_basis'} and Lemma~\ref{lemma:bases}, we obtain the following.
\begin{lemma}
\label{lemma:lms_basis2}
	Let $g:U\to\C^p$ be a continuous function, where $U\subset\C^n$ is an open subset. Let $\mathcal{B}$ be a basis for the Euclidean topology of $U$ and $\mathcal{B}'$ be a basis for the Euclidean topology of $\C^p$, both given by relatively compact open sets. If $\Gamma(g)\subset U\times\C^p$ is not a $q$-local maximum set, then there exist two basic open sets $B\in\mathcal{B}$ and $B'\in\mathcal{B}'$, and a continuous $q$-plurisubharmonic function $\rho:V\to\R$, where $V\subset U\times\C^p$ is a neighbourhood of $K\coloneqq\overline{B\times B'}$, such that
	\begin{align*}
		\max_{\Gamma(g)\cap K} \rho &> \max_{\Gamma(g)\cap bK} \rho \\
		g(\overline{B}) &\subset B'
	\end{align*}
\end{lemma}
\begin{proof}
	By Lemma~\ref{lemma:lms_basis'}, there exists a point $z_0\in\Gamma(g)$ such that: for every neighbourhood $V\subset U\times\C^p$ of $z_0$ there exists a continuous $q$-plurisubharmonic function $\rho:V'\to\R$, where $V'\subset U\times\C^p$ is a neighbourhood of $\Gamma(g)\cap V$, and a point $w_0\in\Gamma(g)\cap V$ such that
	\begin{gather*}
		\rho(w_0)=0 \\
		\rho(w)<0 \quad \forall w\in (\Gamma(g)\cap V)\smallsetminus\{w_0\}
	\end{gather*}
	Fix $V$ relatively compact and take the respective $\rho$ and $w_0$. Without loss of generality, we can assume that $V'=V$.
	
	By Lemma~\ref{lemma:bases}, there exist two basic open sets $B\in\mathcal{B}$ and $B'\in\mathcal{B}'$ such that $w_0\in B\times B'\subset V$ and $g(\overline{B})\subset B'$. Since the basic open sets can be taken arbitrarily small, we can assume that $K\coloneqq\overline{B\times B'}\subset V$. By construction, $w_0\in\mathring{K}$ is a strict global maximum point for $\rho|_{\Gamma(g)\cap K}$.
\end{proof}

Before concluding this section, we prove that a real vector subspace with the local maximum property has a certain form.
\begin{lemma}
\label{lemma:type}
	Let $L\subset\C^N$ be a $(2n+1)$-dimensional real vector subspace with the $(n-1)$-local maximum property, for $1\leq n<N$. Then $L$ is of type $\C^n\times\R$.
\end{lemma}
\begin{proof}
	Let $H\subset\C^N$ be the maximal complex vector subspace contained in $L$ and let $k=\dim_\C(H)$. By contradiction, suppose that $k<n$.  
	
	Take a $(N-k)$-dimensional complex vector subspace $\tilde{H}\subset\C^N$ such that
	\begin{equation}
	\label{eq:orthogonal}
		\tilde{H}\cap H=\{0\}
	\end{equation}
	Let $\tilde{L}\coloneqq L\cap\tilde{H}$. Since $\codim_\C(\tilde{H})=k\leq n-1$, then $\tilde{L}$ has the local maximum property. Moreover, note that
	\[
		\dim_\R(\tilde{L})=(2n+1)+2(N-k)-2N=2(n-k)+1\eqqcolon2r+1\geq3
	\]
	By~\eqref{eq:orthogonal}, $\tilde{L}$ does not contain complex lines and thus $\tilde{L}\cap i\tilde{L}=\{0\}$. Let 
	\[
		\mathcal{B}\coloneqq\{v_1,\dots,v_{2r+1}\}
	\]    
	be a real basis of $\tilde{L}$. Then $\mathcal{B}$ is a complex basis of $\tilde{L}+i\tilde{L}$ and we can extend $\mathcal{B}$ to a complex basis $\tilde{\mathcal{B}}\coloneqq\{v_1,\dots,v_N\}$ of $\C^N$. For each $z\in\C^N$, we can write
	\[
		z=\sum_{j=1}^N z_jv_j
	\]
	for some $z_j\in\C$. Let us define
	\[
		f:\C^N\to\C, \quad f(z)=-\sum_{j=1}^{2r+1} z_j^2
	\]
	Note that $f$ is holomorphic and thus the function $u\coloneqq\Re(f)$ is pluriharmonic on $\C^N$. By construction, for each $z\in\tilde{L}$ we have that
	\[
		z=\sum_{j=1}^{2r+1} x_jv_j
	\]
	for some $x_j\in\R$ and hence
	\[
		u(z)=-\sum_{j=1}^{2r+1} x_j^2 \quad \forall z\in\tilde{L}
	\]
	Thus $u$ is a pluriharmonic function on $\tilde{L}$ which has a strict maximum at $0$ and therefore $\tilde{L}$ does not have the local maximum property. This is a contradiction.
\end{proof}

\section{Upper closed limit and local maximum property}
\label{sec:limit}
In this section, we prove that the upper closed limit of a net preserves the local maximum property.

First of all, let us recall some definitions.
\begin{definition}
	Let $(\Lambda,\leq)$ be a directed set and let $\Lambda'\subset\Lambda$ be a subset. We say that $\Lambda'$ is \emph{cofinal} if for every $\lambda\in\Lambda$ there exists $\lambda'\in\Lambda'$ such that $\lambda\leq\lambda'$. We say, instead, that $\Lambda'$ is \emph{residual} if there exists $\lambda_0\in\Lambda$ such that
	\[
		\{\lambda\in\Lambda \mid \lambda\geq\lambda_0\}\subset\Lambda'
	\]
\end{definition}
\begin{definition}
	Let $X$ be a topological space and let $\{A_\lambda\}_{\lambda\in\Lambda}$ be a net of closed subsets of $X$. A point $x\in X$ is called \emph{cluster point} (resp., \emph{limit point}) of $\{A_\lambda\}$ if for every neighbourhood $U\subset X$ of $x$ there exists a cofinal (resp., residual) subset $\Lambda'\subset\Lambda$ such that $U\cap A_\lambda\neq\emptyset$ for every $\lambda\in\Lambda'$.
\end{definition}
\begin{definition}
	The \emph{upper closed limit} of $\{A_\lambda\}$ is
	\[
		\limsup_{\lambda\in\Lambda} A_\lambda=\bigl\{x\in X \mid \text{$x$ is a cluster point of $\{A_\lambda\}$}\bigr\}
	\]
	The \emph{lower closed limit} of $\{A_\lambda\}$ is
	\[
		\liminf_{\lambda\in\Lambda} A_\lambda=\bigl\{x\in X \mid \text{$x$ is a limit point of $\{A_\lambda\}$}\bigr\}
	\]
\end{definition}

Let us also recall the following characterisations (see~\cite[Proposition~5.2.2]{beer:closedsets}).
\begin{proposition}
	Let $X$ be a Hausdorff space and let $\{A_\lambda\}_{\lambda\in\Lambda}$ be a net of closed subsets. Then
	\begin{align*}
		\limsup_{\lambda\in\Lambda} A_\lambda &= \bigcap_{\substack{\Lambda'\subset\Lambda \\ \text{\emph{residual}}}}\overline{\bigcup_{\lambda\in\Lambda'} A_\lambda} \\
		\liminf_{\lambda\in\Lambda} A_\lambda &= \bigcap_{\substack{\Lambda'\subset\Lambda \\ \text{\emph{cofinal}}}}\overline{\bigcup_{\lambda\in\Lambda'} A_\lambda}
	\end{align*}
\end{proposition}

Our goal is to prove that the upper closed limit of a net of $q$-local maximum sets in $\C^N$ is a closed $q$-local maximum set. In order to do so, we proceed by steps. The first step is to prove that the union and the intersection (the latter, with some further conditions) of a family of $q$-local maximum sets remain $q$-local maximum sets.
\begin{proposition}
\label{prop:union}
	Let $U\subset\C^N$ be an open subset and let $\{A_\lambda\}_{\lambda\in\Lambda}$ be a family of closed subsets of $U$. If $A_\lambda$ is a $q$-local maximum set for every $\lambda\in\Lambda$, then
	\[
		A\coloneqq\overline{\bigcup_{\lambda\in\Lambda}A_\lambda}^U
	\]
	is a $q$-local maximum set, where $\overline{\cdot}^U$ denotes the closure in $U$.
\end{proposition}
\begin{proof}
	Let $L\subset\C^N$ be a complex affine subspace of codimension $q$. To conclude, it is sufficient to show that $A\cap L$ is a local maximum set, if non-empty.
	
	Let $a\in A\cap L$, then there exists a sequence $\{a_n\}\subset\bigcup_{\lambda} A_\lambda\subset U$ which converges to $a$. Note that for every $n$ there exists $\lambda_n\in\Lambda$ such that $a_n\in A_{\lambda_n}$. Let $L_0\coloneqq L-a$ be the complex vector subspace associated to $L$. Then $a_n\in(L_0+a_n)\cap A_{\lambda_n}\subset U$. Let
	\[
		Z_a\coloneqq\limsup_{n\to+\infty}\,\!\!^U(L_0+a_n)\cap A_{\lambda_n}
	\]
	where $\limsup\,\!\!^U$ denotes the upper closed limit with respect to $U$. Since every $A_{\lambda_n}$ is a $q$-local maximum set, then every $(L_0+a_n)\cap A_{\lambda_n}$ has the local maximum property and thus $Z_a$ is a local maximum set, by~\cite[Proposition~1.2(d)]{slodkowski:pseudoconcave}. Clearly, $Z_a\subset A$. Moreover,
	\[
		\limsup_{n\to+\infty}(L_0+a_n)=L
	\]
	and thus $Z_a\subset A\cap L$. Since $a\in A\cap L$ was arbitrary, we have that
	\[
		A\cap L=\bigcup_{a\in A\cap L} Z_a
	\]
	Note that $A\cap L$ is closed in $U$ and hence we can write
	\[
		A\cap L=\overline{\bigcup_{a\in A\cap L} Z_a}^U
	\]
	Since this is the closure of a union of local maximum sets, then it is a local maximum set, by~\cite[Proposition~1.2(e)]{slodkowski:pseudoconcave}. Thus $A$ has the $q$-local maximum property.
\end{proof}

\begin{lemma}
\label{lemma:limsup_inter}
	Let $X$ be a Hausdorff space and let $\{A_\lambda\}_{\lambda\in\Lambda}$ be a net of closed subsets of $X$. If $A_{\lambda'}\subset A_\lambda$ whenever $\lambda'\geq\lambda$, then
	\[
		\limsup_{\lambda\in\Lambda} A_\lambda=\bigcap_{\lambda\in\Lambda} A_\lambda
	\]
\end{lemma}
\begin{proof}
	If $x\in\bigcap_{\lambda} A_\lambda$, then for every neighbourhood $V\subset X$ of $x$ it is clear that $x\in A_\lambda\cap V\neq\emptyset$ for every $\lambda\in\Lambda$. Since $\Lambda$ is trivially a cofinal subset of itself, we conclude that $x\in\limsup_\lambda A_\lambda$.
	
	Viceversa, let
	\[
		x\in\limsup_{\lambda\in\Lambda} A_\lambda=\bigcap_{\substack{\Lambda'\subset\Lambda \\ \text{residual}}}\overline{\bigcup_{\lambda\in\Lambda'} A_\lambda}
	\]
	Fix $\lambda_0\in\Lambda$. To conclude, we need to show that $x\in A_{\lambda_0}$. Since $\Lambda$ is a directed set, then there exists $\lambda'\in\Lambda$ such that $\lambda'\geq\lambda_0$. Set
	\[
		\Lambda'\coloneqq\{\lambda\in\Lambda \mid \lambda\geq\lambda'\}
	\]
	Then $\Lambda'\subset\Lambda$ is a residual subset and thus
	\[
		x\in\overline{\bigcup_{\lambda\in\Lambda'} A_\lambda}
	\]
	Hence there exists a sequence $\{x_n\}\subset\bigcup_{\lambda\in\Lambda'} A_\lambda$ which converges to $x$. Note that for every $n$ there exists $\lambda_n\in\Lambda'$ such that $x_n\in A_{\lambda_n}$. By construction, we have that, for every $n$, $\lambda_n\geq\lambda'\geq\lambda_0$ and thus $A_{\lambda_n}\subset A_{\lambda_0}$, by hypothesis. In particular, $x_n\in A_{\lambda_0}$ for every $n$. Therefore
	\[
		x=\lim_{n\to+\infty} x_n\in A_{\lambda_0}
	\]
	since $A_{\lambda_0}$ is closed.
\end{proof}

\begin{proposition}
\label{prop:intersection}
	Let $U\subset\C^N$ be an open subset and let $\{A_\lambda\}_{\lambda\in\Lambda}$ be a net of closed subsets of $U$, such that $\bigcap_\lambda A_\lambda$ is non-empty. Assume that $A_{\lambda'}\subset A_\lambda$ whenever $\lambda'\geq\lambda$ and that every $A_\lambda$ is a $q$-local maximum set. Then
	\[
		A\coloneqq\bigcap_{\lambda\in\Lambda} A_\lambda
	\]
	is a $q$-local maximum set.
\end{proposition}
\begin{proof}
	Let $L\subset\C^N$ be a complex affine subspace of codimension $q$ intersecting $A$. Then
	\[
		A\cap L=\bigcap_{\lambda\in\Lambda} A_\lambda\cap L
	\]
	Note that every $A_\lambda\cap L$ is a local maximum set. Since $A_{\lambda'}\cap L\subset A_\lambda\cap L$ whenever $\lambda'\geq\lambda$, then
	\[
		\bigcap_{\lambda\in\Lambda} A_\lambda\cap L=\limsup_{\lambda\in\Lambda} A_\lambda\cap L
	\]
	by Lemma~\ref{lemma:limsup_inter}. By~\cite[Proposition~1.2(d)]{slodkowski:pseudoconcave}, the right hand member is a local maximum set and thus so is $A\cap L$. Therefore $A$ is a $q$-local maximum set.
\end{proof}

The second step is to prove that the upper closed limit preserves the local maximum property when all the $q$-local maximum sets considered are closed. In order to do so, we use another characterisation of this limit.
\begin{definition}
	Let $X$ be a Hausdorff space and let $\{A_\lambda\}_{\lambda\in\Lambda}$ be a net of closed subsets of $X$. We say that $\{A_\lambda\}$ \emph{escapes to infinity} if for every compact $K\subset X$ there exists a residual subset $\Lambda'\subset\Lambda$ such that 
	\[
		A_\lambda\cap K=\emptyset \quad \forall\lambda\in\Lambda'
	\]
\end{definition}
\begin{remark}
\label{rem:limsup}
	We have that $\limsup_\lambda A_\lambda$ is non-empty if and only if $\{A_\lambda\}$ does not escape to infinity.
\end{remark}
\begin{remark}
\label{rem:escaping}
	Let $\{A_\lambda\}_{\lambda\in\Lambda}$ be a net of subsets of $\C^N$. Then $\{A_\lambda\}$ escapes to infinity if and only if
	\[
		\lim_{\lambda\in\Lambda} \operatorname{dist}(0,A_\lambda)=+\infty
	\]
\end{remark}

\begin{lemma}
\label{lemma:characterisation_new}
	Let $X$ be a Hausdorff space and let $\{A_\lambda\}_{\lambda\in\Lambda}$ be a net of closed subsets of $X$. Then
	\[
		\limsup_{\lambda\in\Lambda} A_\lambda=\bigcap_{\lambda_0\in\Lambda}\overline{\bigcup_{\lambda\geq\lambda_0} A_\lambda}
	\]
\end{lemma}
\begin{proof}
	Recall that
	\[
		A\coloneqq\limsup_{\lambda\in\Lambda} A_\lambda=\bigcap_{\substack{\Lambda'\subset\Lambda \\ \text{residual}}}\overline{\bigcup_{\lambda\in\Lambda'} A_\lambda}
	\]
	and set
	\[
		A_0\coloneqq\bigcap_{\lambda_0\in\Lambda}\overline{\bigcup_{\lambda\geq\lambda_0} A_\lambda}
	\]
	Let $x\in A$ and fix $\lambda_0\in\Lambda$. If we set
	\[
		\Lambda_0\coloneqq\{\lambda\in\Lambda \mid \lambda\geq\lambda_0\}
	\]
	then $\Lambda_0$ is a residual subset of $\Lambda$.
	Hence
	\[
		x\in\overline{\bigcup_{\lambda\in\Lambda_0} A_\lambda}
	\]
	By the arbitrariness of $\lambda_0$, we obtain that $A\subset A_0$. Note that this argument also shows that if $A_0$ is empty, then so is $A$.
	
	Viceversa, let $x\in A_0$ and fix $\Lambda'\subset\Lambda$ residual. Then there exists $\lambda'\in\Lambda$ such that
	\[
		\{\lambda\in\Lambda \mid \lambda\geq\lambda'\}\subset\Lambda'
	\]
	Hence
	\[
		x\in\overline{\bigcup_{\lambda\geq\lambda'} A_\lambda}\subset\overline{\bigcup_{\lambda\in\Lambda'} A_\lambda}
	\]
	By the arbitrariness of $\Lambda'$, we obtain that $A_0\subset A$. Again, note that this argument also shows that if $A$ is empty, then so is $A_0$.
	
	Therefore, in all cases, $A=A_0$.
\end{proof}

\begin{proposition}
\label{prop:limsup1}
	Let $U\subset\C^N$ be an open subset and let $\{A_\lambda\}_{\lambda\in\Lambda}$ be a net of closed subsets of $U$. Assume that every $A_\lambda$ is a $q$-local maximum set and that $\{A_\lambda\}$ does not escape to infinity (relatively to $U$). Then
	\[
		A\coloneqq\limsup_{\lambda\in\Lambda}\,\!\!^U A_\lambda
	\]
	is a closed $q$-local maximum set.
\end{proposition}
\begin{proof}
	First of all, note that $A$ is non-empty by Remark~\ref{rem:limsup}. By Lemma~\ref{lemma:characterisation_new}, we have that
	\[
		A=\bigcap_{\lambda_0\in\Lambda}\overline{\bigcup_{\lambda\geq\lambda_0} A_\lambda}^U
	\]
	Note that
	\[
		\overline{\bigcup_{\lambda\geq\lambda'_0} A_\lambda}^U\subset\overline{\bigcup_{\lambda\geq\lambda_0} A_\lambda}^U \quad \text{whenever }\lambda'_0\geq\lambda_0
	\]
	Therefore, by Proposition~\ref{prop:union} and Proposition~\ref{prop:intersection}, we conclude.
\end{proof}

Now we can finally prove the main statement of this section.
\begin{proposition}
\label{prop:limsup}
	Let $\{A_\lambda\}_{\lambda\in\Lambda}$ be a net of $q$-local maximum sets in $\C^N$. For every $\lambda\in\Lambda$, let $B_\lambda\coloneqq\overline{A_\lambda}\smallsetminus A_\lambda$ and assume that $\{B_\lambda\}$ escapes to infinity but $\{A_\lambda\}$ does not. Then
	\[
		A\coloneqq\limsup_{\lambda\in\Lambda} A_\lambda
	\]
	is a closed subset of $\C^N$ with the $q$-local maximum property.
\end{proposition}
\begin{proof}
	Let $B_R\subset\C^N$ be an open ball of radius $R>0$ centred at the origin. To conclude, it is sufficient to show that $A\cap B_R$ is a $q$-local maximum set, provided it is non-empty. Note that the non-emptiness assumption happens if we choose $R>0$ large enough such that $\overline{B_R}$ intersects infinitely many $A_\lambda$'s, which is possible since $\{A_\lambda\}$ does not escape to infinity.
	
	By Lemma~\ref{lemma:characterisation_new}, we have that
	\[
		B_R\cap A=\bigcap_{\lambda\in\Lambda}\biggl(B_R\cap\overline{\bigcup_{\lambda'\geq\lambda} A_{\lambda'}}\biggr)
	\]
	Choose $\lambda_0$ such that $\overline{B_R}\cap B_\lambda=\emptyset$ for every $\lambda\geq\lambda_0$, which exists since $\{B_\lambda\}$ escapes to infinity. Note that
	\[
		B_R\cap A=\bigcap_{\lambda\geq\lambda_0}\biggl(B_R\cap\overline{\bigcup_{\lambda'\geq\lambda} A_{\lambda'}}\biggr)
	\]
	We claim that
	\begin{equation*}
	\label{eq:claim'} \tag{$\star$}
		B_R\cap\overline{\bigcup_{\lambda'\geq\lambda} A_{\lambda'}}=\overline{\bigcup_{\lambda'\geq\lambda} B_R\cap A_{\lambda'}}^{B_R} \quad \forall\lambda\geq\lambda_0
	\end{equation*}
	where $\overline{\cdot}^{B_R}$ denotes the closure in $B_R$. Fix $\lambda\geq\lambda_0$. Clearly, we have that
	\[
		\overline{\bigcup_{\lambda'\geq\lambda} B_R\cap A_{\lambda'}}^{B_R}=\overline{B_R\cap\bigcup_{\lambda'\geq\lambda} A_{\lambda'}}^{B_R}\subset \overline{B_R}^{B_R}\cap\overline{\bigcup_{\lambda'\geq\lambda} A_{\lambda'}}^{B_R}\subset B_R\cap\overline{\bigcup_{\lambda'\geq\lambda} A_{\lambda'}}
	\]
	Viceversa, let
	\[
		z\in B_R\cap\overline{\bigcup_{\lambda'\geq\lambda} A_{\lambda'}}
	\]
	Then there exists a sequence $\{z_n\}\subset\bigcup_{\lambda'\geq\lambda} A_{\lambda'}$ which converges to $z$. Since $z\in B_R$, then there exists $n_0$ such that $z_n\in B_R$ for every $n\geq n_0$. Hence
	\[
		z_n\in B_R\cap\bigcup_{\lambda'\geq\lambda} A_{\lambda'}=\bigcup_{\lambda'\geq\lambda} B_R\cap A_{\lambda'} \quad \forall n\geq n_0
	\]
	and thus
	\[
		z=\lim_{n\to+\infty} z_n\in\overline{\bigcup_{\lambda'\geq\lambda} B_R\cap A_{\lambda'}}^{B_R}
	\]
	Therefore \eqref{eq:claim'} holds. Thus
	\begin{equation}
	\label{eq:limsup_ball'}
		B_R\cap A=\bigcap_{\lambda\geq\lambda_0}\overline{\bigcup_{\lambda'\geq\lambda} B_R\cap A_{\lambda'}}^{B_R}=\limsup_{\lambda\geq\lambda_0}\,\!\!^{B_R} (B_R\cap A_\lambda)
	\end{equation}
	Recall that $B_R\cap B_\lambda=\emptyset$ for all $\lambda\geq\lambda_0$. Hence, for every $\lambda\geq\lambda_0$,
	\[
		B_R\cap A_\lambda=(B_R\cap A_\lambda)\cup(B_R\cap B_\lambda)=B_R\cap(A_\lambda\cup B_\lambda)=B_R\cap\overline{A_\lambda}
	\]
	and thus $A_\lambda$ is relatively closed in $B_R$. Since every $A_\lambda\cap B_R$ is a closed $q$-local maximum set in $B_R$, then $A\cap B_R$ is a $q$-local maximum set, by~\eqref{eq:limsup_ball'} and Proposition~\ref{prop:limsup1}.
\end{proof}

\section{The higher codimension case}
\label{sec:second_gen}
In this section, we prove the Main Theorem.

Following Pawlaschyk and Shcherbina's idea, which generalises Chirka's Theorem to higher codimension graphs, we want to generalise Corollary~\ref{coro:foliation2} to the case of an $n$-pseudoconcave subset of $\C^N$, for some $N>n$. To achieve this, the use of the local maximum property will be fundamental.

The idea is to reduce our problem to the case of an hypersurface, then use Corollary~\ref{coro:foliation} to obtain a foliation on it and, finally, lift this foliation to the original graph. To achieve the first step, we will use the following reduction lemma.
\begin{lemma}
\label{lemma:reduction}
    For $n\geq1$, let $D\subset\C^n\times\R$ be a domain and let $X_0\subset D$ be a closed subset. For $N>n$ and $p=N-n-1\geq0$, let $g=(g',g''):X_0\to\R\times\C^p$ be a continuous function and let
    \[
        Z_0\coloneqq\Gamma(g)=\bigl\{(z,w,\zeta)\in X_0\times i\R\times\C^p \mid \bigl(\Im(w),\zeta\bigr)=g\bigl(z,\Re(w)\bigr)\bigr\}
    \]
    be its graph in $\C^N=\C^n_z\times(\R\times i\R)_w\times\C^p_\zeta$. Choose $r\in\{0,\dots,p\}$ and let
    \[
        \pi:D\times i\R\times\C^p\to D\times i\R\times\C^r, \quad \pi(z,w,\zeta)=(z,w,\zeta_{\mu_1},\dots,\zeta_{\mu_r})
    \]    
    be the projection, for $1\leq\mu_1<\dots<\mu_r\leq p$. If $Z_0$ is a $q$-local maximum set, then $Z'_0\coloneqq\pi(Z_0)$ is a $q$-local maximum set.
\end{lemma}
\begin{proof}
    Since we want to prove a local property, we can assume, at most shrinking $D$ and applying a holomorphic change of coordinates, that $\mu_j=j$ for all $j=1,\dots,r$.

    By contradiction, suppose that there exists a compact $K\subset D\times i\R\times\C^r$, a smooth $q$-plurisub\-harmonic function $\phi$, defined on a neighbourhood $U\subset D\times i\R\times\C^r$ of $K$, and a point $\eta_0\in Z'_0\cap\mathring{K}$ such that
    \begin{equation}
    \label{eq:contradiction}
        \phi(\eta_0)=\max_{Z'_0\cap K}\phi>\max_{Z'_0\cap bK}\phi
    \end{equation}
    Let us define the function
    \[
        \tilde\phi\coloneqq\phi\circ\pi:\pi^{-1}(U)=U\times\C^{p-r}\to\R
    \]
    We want to prove the $\tilde\phi$ is a smooth $q$-plurisubharmonic function. Note that
    \begin{equation}
    \label{eq:projection}
        \tilde\phi(z,w,\zeta)=\phi(z,w,\zeta_1,\dots,\zeta_r)
    \end{equation}
    and thus $\tilde\phi$ is a smooth function. Let
    \[
        \operatorname{L}_{\tilde\phi}=\biggl(\frac{\partial^2\tilde\phi}{\partial\xi_i \partial\bar{\xi}_j}\biggr)_{ij}\in\Mat_{\mathscr{C}^\infty(U\times\C^{p-r})}(N\times N)
    \]
    be the Levi matrix of $\tilde\phi$. Using~\eqref{eq:projection}, we obtain that
    \begin{equation}
    \label{eq:matrix}
        \operatorname{L}_{\tilde\phi}=
        \begin{pmatrix}
            \operatorname{L}_\phi & \vec{0}_{(n+1+r)\times(p-r)} \\
            \vec{0}_{(p-r)\times(n+1+r)} & \vec{0}_{(p-r)\times(p-r)}
        \end{pmatrix}
    \end{equation}
    where $\operatorname{L}_\phi\in\Mat_{\mathscr{C}^\infty(U)}((n+1+r)\times(n+1+r))$ is the Levi matrix of $\phi$. If $\sigma(\operatorname{L}_\phi(\xi))=\{\alpha_1,\dots,\alpha_{n+1+r}\}$ is the spectrum of $\operatorname{L}_\phi(\xi)$ (with repetitions allowed), for a fixed point $\xi\in U$, then
    \[
        \sigma(\operatorname{L}_{\tilde\phi}(\xi))=\{\alpha_1,\dots,\alpha_{n+1+r},0,\dots,0\}
    \]
    by~\eqref{eq:matrix}. Since $\phi$ is $q$-plurisubharmonic, then $\sigma(\operatorname{L}_\phi(\xi))$ has at most $q$ negative elements and thus $\sigma(\operatorname{L}_{\tilde\phi}(\xi))$ has at most $q$ negative elements. Therefore $\tilde\phi$ is a smooth $q$-plurisubharmonic function.

    Note that
    \begin{align*}
        Z_0 &= \bigl\{\bigl(x,ig'(x),g''(x)\bigr) \mid x\in X_0\bigr\} \\
        Z'_0 &= \bigl\{\bigl(x,ig'(x),g''_1(x),\dots,g''_r(x)\bigr) \mid x\in X_0\bigr\}
    \end{align*}
    Since $\eta_0\in Z'_0$, then $\eta_0=(x_0,ig'(x_0),g''_1(x_0),\dots,g''_r(x_0))$ for some $x_0\in X_0$ and thus
    \[
        \pi^{-1}(\eta_0)\cap Z_0=\{(x_0,ig'(x_0),g''(x_0)\}=\{(\eta_0,\chi_0)\}
    \]
    where $\chi_0\coloneqq(g''_{r+1}(x_0),\dots,g''_p(x_0))$. By~\eqref{eq:contradiction}, we obtain that
    \begin{equation}
    \label{eq:maxpoint}
        \tilde\phi(\eta_0,\chi_0)=\phi(\eta_0)>\phi(\eta)=\phi(\eta,\chi)
    \end{equation}
    for all $(\eta,\chi)\in\pi^{-1}(Z'_0\cap bK)$.

    Take $C\subset\C^N$ compact such that $(\eta_0,\chi_0)\in Z_0\cap\mathring{C}$, $\pi(C)=K$ and $Z_0\cap bC= Z_0\cap\pi^{-1}(bK)$. Since $Z_0\cap bC\subset\pi^{-1}(Z'_0\cap bK)$, then
    \[
        \tilde\phi(\eta_0,\chi_0)>\max_{Z_0\cap bC}\tilde\phi
    \]
    by~\eqref{eq:maxpoint}. This is a contradiction with the $q$-local maximum property of $Z_0$. Therefore $Z'_0$ is a $q$-local maximum set.
\end{proof}

Now we can state and prove our main theorem. Let $Z\subset\C^N$ be an $n$-pseudo\-concave subset, for $1\leq n<N$. Assume that for every $\xi\in Z$ there exists a neighbourhood $U\subset\C^N$ of $\xi$ and a complex linear coordinate system with respect to which $Z_0\coloneqq Z\cap U$ is the graph of a continuous function 
\[
    g=(g',g''):X_0\to \R\times\C^p
\]    
with $X_0\subset\C^n\times\R$ relatively closed subset of $\pi(U)$ (where $\pi:\C^N\to\C^n\times\R$ is the projection), $p=N-n-1\geq0$ and with the identification $\C^N=\C^n_z\times(\R\times i\R)_w\times\C^p_\zeta$. Namely,
\begin{equation}
\label{eq:locmaxsetform}
    Z_0=\Gamma(g)=\bigl\{(z,w,\zeta)\in X_0\times i\R\times\C^p \mid \bigl(\Im(w),\zeta\bigr)=g\bigl(z,\Re(w)\bigr)\bigr\}
\end{equation}
At most shrinking $U$, we can assume that $B\coloneqq\pi(U)$ is an open ball containing $X_0$.
\begin{theorem}
\label{thm:foliation'}
    Let $Z\subset\C^N$ be an $n$-pseudoconcave subset as above. Then $Z$ is foliated by $n$-dimensional complex manifolds.
\end{theorem}
\begin{proof}
    Fix $\xi_0\in Z$ and take $U\subset\C^N$, $X_0\subset\C^n\times\R$ and $g:X_0\to\R\times\C^p$ such that $Z_0\coloneqq Z\cap U$ has the form \eqref{eq:locmaxsetform} and suppose that $B\coloneqq\pi(U)$ is an open ball containing $X_0$. By Tietze Extension Theorem, we can see $g$ as a continuous function on all $B$.
    
    Since $Z$ is an $(n-1)$-local maximum set and $Z_0$ is open in it, then $Z_0$ has the $(n-1)$-local maximum property. By Lemma~\ref{lemma:reduction} (with $r=0$), $Z'_0\coloneqq\Gamma(g')$ is an $(n-1)$-local maximum set, i.e., it is $n$-pseudoconcave. Thus, by Corollary~\ref{coro:foliation}, $Z'_0$ is foliated by a (unique) family $\{\ell_\alpha\}_{\alpha\in A}$ of $n$-dimensional complex manifolds.

    Let us define
    \[
        g_\alpha:\ell_\alpha\to\C^p, \quad g_\alpha=g''\circ\varpi
    \]
    where $\varpi:B\times i\R\to B$ is the projection. Recall that $\ell_\alpha=\Gamma(f_\alpha)$, where $f_\alpha:G_\alpha\to\C$ is a holomorphic function and $G_\alpha\subset\C^n$ is a contractible domain of holomorphy. Then
    \[
        \Gamma(g_\alpha)=\bigl\{\bigl(x,ig'(x),g''(x)\bigr)\in X_0\times i\R\times\C^p \mid x\in\varpi(\ell_\alpha)\subset X_0\bigr\}\subset Z_0
    \]
    where $\varpi(\ell_\alpha)=\{(z,\Re(f_\alpha(z)))\in X_0 \mid z\in G_\alpha\}$. We want to show that the $g_\alpha$'s are holomorphic functions.

    First of all, recall that 
    \[
    	Z'_0=\bigsqcup_{\alpha\in A} \ell_\alpha
    \]
    If $I$ is a connected component of $A$, then $I$ is either a single point or a non-open interval. Moreover, we have that
    \[
        Z'_{0,I}\coloneqq\bigsqcup_{\alpha\in I} \ell_\alpha \qquad Z_{0,I}\coloneqq\bigsqcup_{\alpha\in I} \Gamma(g_\alpha)
    \]
    are connected components of $Z'_0$ and $Z_0$, respectively, and thus $(n-1)$-local maximum sets.

    By contradiction, suppose that there exists $\alpha_0\in I$ such that $g_{\alpha_0}$ is not holomorphic in a neighbourhood of a point $(z_0,w_0)\in\ell_{\alpha_0}$. In particular, there exists $j\in\{1,\dots,p\}$ such that the $j^{\text{th}}$ component $g_{\alpha_0,j}$ of $g_{\alpha_0}$ is not holomorphic in a neighbourhood of $(z_0,w_0)$.

    Let us suppose, at most applying a holomorphic change of coordinates, that $\ell_{\alpha_0}$ has the form $\Delta_r^n\times\{w=0\}$ near $(z_0,w_0)$, where $\Delta_r^n\coloneqq\Delta_r^n(z_0)\subset\C^n$ is an open ball of radius $r$ centred at $z_0$. Note that we can find an open subset $G\subset\C^n$ such that $G\subset G_\alpha$ for all $\alpha\in J$, where $J\subset I$ is an open interval (in $I$) containing $\alpha_0$. At most shrinking $r$, we can assume that $\Delta_r^n\subset G$ and that $g_{\alpha_0,j}$ is not holomorphic on $\Delta_r^n\times\{0\}$. Let us identify the restriction of $g_{\alpha_0,j}$ to $\Delta_r^n\times\{0\}$ with the function
    \[
        \tilde{g}_{\alpha_0,j}:\Delta_r^n\to\C, \quad \tilde{g}_{\alpha_0,j}(z)=g_{\alpha_0,j}(z,0)
    \]
    Since $\tilde{g}_{\alpha_0,j}$ is not holomorphic, then, by a classical result due to Hartogs (see~\cite[Theorem~1.1]{pawl_shcher:foliations}), $(\Delta_r^n\times\C)\smallsetminus\Gamma(\tilde{g}_{\alpha_0,j})$ is not pseudoconvex and thus $\tilde\Gamma_{\alpha_0,j}\coloneqq\Gamma(\tilde{g}_{\alpha_0,j})$ is not an $(n-1)$-local maximum set.

    Let
    \[
        \mathcal{B}\coloneqq\{\Delta_d^n(z)\times\Delta_l(\zeta_j) \mid d,l>0, \,\, z\in\Delta_r^n, \,\, \zeta_j\in\C\}
    \]
    be a basis for the Euclidean topology of $\Delta_r^n\times\C$ given by relatively compact open sets, where $\Delta_d^n(z)\subset\Delta_r^n$ is an open ball centred at $z$ and $\Delta_l(\zeta_j)\subset\C$ is an open disc centred at $\zeta_j$. By Lemma~\ref{lemma:lms_basis2}, there exists a basic open set
    \[
        \Delta_{r'}^n\times\Delta_\sigma\coloneqq\Delta_{r'}^n(\tilde{z})\times\Delta_\sigma(\tilde\zeta_j)\in\mathcal{B}
    \]
    a continuous $(n-1)$-plurisubharmonic function $\rho:V\to\R$, where $V\subset\Delta_r^n\times\C$ is a neighbourhood of $K\coloneqq\overline{\Delta_{r'}^n\times\Delta_\sigma}$, and a point $(\bar{z},\bar\zeta_j)\in\tilde\Gamma_{\alpha_0,j}\cap\mathring{K}$ such that
    \begin{gather}
    	0=\rho(\bar{z},\bar\zeta_j)=\max_{\tilde\Gamma_{\alpha_0,j}\cap K} \rho>\max_{\tilde\Gamma_{\alpha_0,j}\cap bK} \rho \notag \\
    	\tilde{g}_{\alpha_0,j}(\overline{\Delta_{r'}^n})\subset\Delta_\sigma \label{eq:inclusion}
    \end{gather}
    Since $\rho$ is continuous, there exists a radius $r''\in(0,r')$ such that
    \begin{equation}
    \label{eq:negativity}
        \rho(z,\zeta_j)<0 \quad \forall(z,\zeta_j)\in\tilde\Gamma_{\alpha_0,j}\cap(\overline{\Delta_{r',r''}^n\times\Delta_\sigma})
    \end{equation}
    where $\Delta_{r',r''}^n\coloneqq\Delta_{r'}^n\smallsetminus\overline{\Delta_{r''}^n}$.

    Let
    \[
        \pi_j:B\times i\R\times\C^p\to B\times i\R\times\C, \quad \pi_j(z,w,\zeta)=(z,w,\zeta_j)
    \]
    be the projection and let $Z_{0,I}^j\coloneqq\pi_j(Z_{0,I})$. By Lemma~\ref{lemma:reduction} (with $r=1$), $Z_{0,I}^j$ is an $(n-1)$-local maximum set. Consider a disc $\Delta_\delta\subset\C$ centred at 0 and set
    \[
        \tilde{V}\coloneqq\{(z,w,\zeta_j)\in B\times i\R\times\C \mid (z,\zeta_j)\in V, \,\, w\in\Delta_\delta\}
    \]
    Since $\tilde{V}$ is open, then
    \[
        \tilde{Z}_{0,I}^j\coloneqq Z_{0,I}^j\cap\tilde{V}=\bigsqcup_{\alpha\in I} \Gamma_{\alpha,j}\cap\tilde{V}
    \]
    is an $(n-1)$-local maximum set, where $\Gamma_{\alpha,j}\coloneqq\Gamma(g_{\alpha,j})$. Moreover, let us set $\ell'_\alpha=\Gamma(f_\alpha|_{\Delta_{r'}^n})$ and $\Gamma'_{\alpha,j}=\Gamma(g_{\alpha,j}|_{\ell'_\alpha})$, for every $\alpha\in J$. Then
    \[
        \Gamma'_j\coloneqq\bigsqcup_{\alpha\in J}\Gamma'_{\alpha,j}\subset Z_{0,I}^j
    \]
    
    Let us define
    \[
        \tilde\rho:\tilde{V}\to\R, \quad \tilde\rho(z,w,\zeta_j)=\rho(z,\zeta_j)
    \]
    Clearly, $\tilde\rho$ is a continuous $(n-1)$-plurisubharmonic function. We want to show that $\tilde\rho(z,w,\zeta_j)<0$ if $(z,w,\zeta_j)\in\overline{\Gamma'_j\cap(\Delta_{r',r''}^n\times\Delta_\delta\times\Delta_\sigma)}$ (at most shrinking $\delta$). First of all, since $\overline{\Delta^n_{r',r''}\times\Delta_\sigma}$ is compact, then $\rho$ is uniformly continuous on it and thus, by~\eqref{eq:negativity}, there exist $t_1,t_2>0$ such that, for all $(z',\zeta'_j)\in\tilde\Gamma_{\alpha_0,j}\cap(\overline{\Delta_{r',r''}^n\times\Delta_\sigma})$, we have that
    \begin{equation}
    \label{eq:temp_negativity}
        \rho(z,\zeta_j)<0 \quad \forall(z,\zeta_j)\in\Delta^n_{t_1}(z')\times\Delta_{t_2}(\zeta'_j)
    \end{equation}
    Moreover, note that if $(z,w,\zeta_j)\in\overline{\Gamma'_j\cap(\Delta_{r',r''}^n\times\Delta_\delta\times\Delta_\sigma)}$, then, in particular, $(z,w,\zeta_j)\in Z^j_0$ and thus $\zeta_j=g''_j(z,\Re(w))$. On the other hand, $\tilde{g}_{\alpha_0,j}(z)\in\Delta_\sigma$, by~\eqref{eq:inclusion}, and thus 
    \begin{equation}
    \label{eq:temp_negativity2}
        (z,g''_j(z,0))=(z,\tilde{g}_{\alpha_0,j}(z))\in\tilde\Gamma_{\alpha_0,j}\cap(\overline{\Delta_{r',r''}^n\times\Delta_\sigma})
    \end{equation}
    Since $\overline{\Delta^n_{r',r''}\times\Delta_\delta}$ is compact, then $g''_j\circ\varpi$ is uniformly continuous on it and thus there exists $\delta'>0$ such that, for all $(z,w),(z',w')\in\overline{\Delta^n_{r',r''}\times\Delta_\delta}$ such that $\norm{(z,w)-(z',w')}<\delta'$, we have that
    \[
        |g''_j(z,\Re(w))-g''_j(z',\Re(w'))|<t_2
    \]
    If $\delta<\delta'$, then, for all $(z,w)\in\overline{\Delta^n_{r',r''}\times\Delta_\delta}$,
    \[
        \norm{(z,w)-(z,0)}=|w|<\delta'
    \]    
    and thus, for all $(z,w,\zeta_j)\in\overline{\Gamma'_j\cap(\Delta_{r',r''}^n\times\Delta_\delta\times\Delta_\sigma)}$,
    \begin{equation}
    \label{eq:temp_negativity3}
        \zeta_j=g''_j(z,\Re(w))\in\Delta_{t_2}(g''_j(z,0))
    \end{equation}
    Therefore, by~\eqref{eq:temp_negativity}, \eqref{eq:temp_negativity2} and \eqref{eq:temp_negativity3}, we obtain that (at most shrinking $\delta$)
    \begin{equation}
    \label{eq:negativity2}
        \tilde\rho(z,w,\zeta_j)<0 \quad \forall(z,w,\zeta_j)\in\overline{\Gamma'_j\cap(\Delta_{r',r''}^n\times\Delta_\delta\times\Delta_\sigma)}
    \end{equation}

    Now recall that, by~\eqref{eq:inclusion},
    \[
        g_{\alpha_0,j}(\overline{\Delta_{r'}^n\times\{w=0\}})\subset\Delta_\sigma
    \]
    Note that, by compactness of $g_{\alpha_0,j}(\overline{\Delta^n_{r'}\times\{0\}})$, there exists $t>0$ such that
    \[
        \Delta_t(g_{\alpha_0,j}(z,0))\subset\Delta_\sigma \quad \forall z\in\overline{\Delta^n_{r'}}
    \]
    Since $\overline{\Delta^n_{r'}\times\Delta_\delta}$ is compact, then $g''_j\circ\varpi$ is uniformly continuous on it and thus there exists $\delta''>0$ such that, for all $(z,w),(z',w')\in\overline{\Delta^n_{r'}\times\Delta_\delta}$ such that $\norm{(z,w)-(z',w')}<\delta''$, we have that
    \[
        |g''_j(z,\Re(w))-g''_j(z',\Re(w'))|<t
    \]
    If $\delta<\delta''$, then, for all $(z,w)\in\overline{\Delta^n_{r'}\times\Delta_\delta}$,
    \[
        \norm{(z,w)-(z,0)}=|w|<\delta''
    \]
    and thus
    \[
        g''_j(z,\Re(w))\in\Delta_t(g''_j(z,0))=\Delta_t(g_{\alpha_0,j}(z,0))\subset\Delta_\sigma
    \]
    Therefore (at most shrinking $\delta$)
    \begin{equation}
    \label{eq:localsubset}
        g''_j\circ\varpi(\overline{\Delta_{r'}^n\times\Delta_\delta})\subset\Delta_\sigma
    \end{equation}
    
    Recall that $I$ is either a single point or a non-open interval. In the former case, $I=\{\alpha_0\}$. In the latter case, $\alpha_0$ is either an interior point or a boundary point for $I$.

	\paragraph{Case {\boldmath $\{\alpha_0\}=I$}.} In this case $\Gamma(g_{\alpha_0})=Z_{0,I}$ and $\Gamma_{\alpha_0,j}=Z_{0,I}^j$ and thus they are $(n-1)$-local maximum sets. Since $\tilde\Gamma_{\alpha_0,j}$ is open in $\Gamma_{\alpha_0,j}$, then it has the $(n-1)$-local maximum property. This is a contradiction. Thus we can assume that $I$ is an interval.

	\paragraph{Case {\boldmath $\alpha_0\in\mathring{I}$}.} Since $Z'_{0,I}$ is a connected component of $Z'_0$ and $J$ is an interval, then there exist $\alpha',\alpha''\in\mathring{J}$ such that $\alpha_0\in(\alpha',\alpha'')$ and 
    \begin{equation}
    \label{eq:local_inclusion}
    	\ell'_\alpha\subset\Delta_{r'}^n\times\Delta_\delta \quad \forall\alpha\in[\alpha',\alpha'']
    \end{equation}
    By~\eqref{eq:localsubset}, we obtain that $g_{\alpha,j}(\ell'_\alpha)\subset\Delta_\sigma$ for all $\alpha\in[\alpha',\alpha'']$. In particular,
    \begin{equation}
    \label{eq:localinclusion}
        \Gamma'_{\alpha,j}\subset\Delta_{r'}^n\times\Delta_\delta\times\Delta_\sigma \quad \forall\alpha\in[\alpha',\alpha'']
    \end{equation}
    Let us define
    \[
        \Omega\coloneqq(\ell_{\alpha'}\cup\ell_{\alpha_0}\cup\ell_{\alpha''})\cap(\Delta_r^n\times\Delta_\delta)
    \]
    and let $h:\Omega\to\C$ be a holomorphic function defined by
    \[
        \begin{cases}
            h\equiv0 & \text{on $\ell_{\alpha'}\cup\ell_{\alpha''}$} \\
            h\equiv1 & \text{on $\ell_{\alpha_0}$}
        \end{cases}
    \]
    Since $\Omega$ is a closed analytic subset of a strictly pseudoconvex domain ($\Delta_r^n\times\Delta_\delta$), then there exists a holomorphic extension $\hat{h}:\Delta_r^n\times\Delta_\delta\to\C$ of $h$. Let us define
    \[
        \hat\rho:\Delta_r^n\times\Delta_\delta\to\R, \quad \hat\rho(z,w)=\log|\hat{h}(z,w)|
    \]
    Then $\hat\rho$ is a pluriharmonic function which satisfies
    \begin{equation}
    \label{eq:cases1}
        \begin{cases}
            \hat\rho\equiv-\infty & \text{on $\ell_{\alpha'}\cup\ell_{\alpha''}$} \\
            \hat\rho\equiv0 & \text{on $\ell_{\alpha_0}$}
        \end{cases}
    \end{equation}
    Now take $\epsilon>0$ and define
    \[
        \psi:\tilde{V}\to\R, \quad \psi(z,w,\zeta_j)=\tilde\rho(z,w,\zeta_j)+\epsilon\hat\rho(z,w)
    \]
    Then $\psi$ is a continuous $(n-1)$-plurisubharmonic function. By construction, we have that
    \begin{equation}
    \label{eq:max_point}
        \psi(\bar{z},0,\bar\zeta_j)=\rho(\bar{z},\bar\zeta_j)+\epsilon\hat\rho(\bar{z},0)=0
    \end{equation}
    On the other hand, if $\epsilon$ is small enough we obtain that
    \begin{equation}
    \label{eq:neg_points}
        \psi(z,w,\zeta_j)<0 \quad \forall(z,w,\zeta_j)\in\overline{\Gamma'_j\cap(\Delta_{r',r''}^n\times\Delta_\delta\times\Delta_\sigma)}\cup\Gamma'_{\alpha',j}\cup\Gamma'_{\alpha'',j}
    \end{equation}
    using \eqref{eq:negativity2} and \eqref{eq:cases1}. Let
    \[
        H\coloneqq\bigsqcup_{\alpha\in(\alpha',\alpha'')} \ell'_\alpha
    \]
    and define $C\coloneqq\Gamma((g''_j\circ\varpi)|_H)$. Equivalently,
    \[
        C=\bigsqcup_{\alpha\in(\alpha',\alpha'')} \Gamma'_{\alpha,j}
    \]
    since $(g''_j\circ\varpi)|_{\ell_\alpha}=g_{\alpha,j}$. By~\eqref{eq:localinclusion}, we have that
    \[
        C\subset Z_{0,I}^j\cap\tilde{V}=\tilde{Z}_{0,I}^j
    \]
    Note that $(\bar{z},0,\bar\zeta_j)\in C$. Indeed, $(\bar{z},0,\bar\zeta_j)\in\Gamma_{\alpha_0,j}$, since $(\bar{z},\bar\zeta_j)\in\tilde\Gamma_{\alpha_0,j}$ and $\ell_{\alpha_0}$ has the local form $\Delta_r^n\times\{w=0\}$. On the other hand, $\bar{z}\in\Delta_{r'}^n$, since $(\bar{z},\bar\zeta_j)\in\mathring{K}$. Thus $(\bar{z},0,\bar\zeta_j)\in C$. On the other hand, note that
    \[
        bC=\overline{C}\smallsetminus C=\biggl(\bigsqcup_{\alpha\in[\alpha',\alpha'']}b\Gamma'_{\alpha,j}\biggr)\cup\Gamma'_{\alpha',j}\cup\Gamma'_{\alpha'',j}
    \]
    where $b\Gamma'_{\alpha,j}=\overline{\Gamma'_{\alpha,j}}\smallsetminus\Gamma'_{\alpha,j}=\Gamma(g_{\alpha,j}|_{b\ell'_\alpha})$ and $b\ell'_\alpha=\overline{\ell'_\alpha}\smallsetminus\ell'_\alpha=\Gamma(f_\alpha|_{b\Delta_{r'}^n})$. For some $\alpha\in[\alpha',\alpha'']$, let $(z,w,\zeta_j)\in b\Gamma'_{\alpha,j}$. Then $z\in b\Delta^n_{r'}\subset\overline{\Delta^n_{r',r''}}$, $w\in\overline{\Delta_\delta}$ (by~\eqref{eq:local_inclusion}) and $\zeta_j\in\Delta_\sigma$ (by~\eqref{eq:localsubset}).
    Hence
    \begin{equation}
    \label{eq:boundary}
        bC\subset\overline{\Gamma'_j\cap(\Delta_{r',r''}^n\times\Delta_\delta\times\Delta_\sigma)}\cup\Gamma'_{\alpha',j}\cup\Gamma'_{\alpha'',j}
    \end{equation}
    Therefore
    \[
        \max_{C}\psi\geq\psi(\bar{z},0,\bar\zeta_j)=0>\max_{bC}\psi
    \]
    by~\eqref{eq:max_point}, \eqref{eq:neg_points} and \eqref{eq:boundary}. Thus $\tilde{Z}_{0,I}^j$ is not an $(n-1)$-local maximum set, since $C$ is a relatively compact and relatively open subset of $\tilde{Z}^j_{0,I}$ and $\psi$ is an $(n-1)$-plurisubharmonic function defined on a neighbourhood $\tilde{V}\subset B\times i\R\times\C$ of $\overline{C}$. This is a contradiction.

	\paragraph{Case {\boldmath $\alpha_0\in bI$}.} Without loss of generality, we can assume that $\alpha_0$ is the left endpoint of $I$. Then $\alpha_0$ is also the left endpoint of $J$. Since $Z'_{0,I}$ is a connected component of $Z'_0$, then there exists $\alpha''\in\mathring{J}$ such that $\ell'_\alpha\subset\Delta_{r'}^n\times\Delta_\delta$ for all $\alpha\in[\alpha_0,\alpha'']$. By~\eqref{eq:localsubset}, we obtain that $g_{\alpha,j}(\ell'_\alpha)\subset\Delta_\sigma$ for all $\alpha\in[\alpha_0,\alpha'']$. In particular,
    \begin{equation}
    \label{eq:localinclusion2}
        \Gamma'_{\alpha,j}\subset\Delta_{r'}^n\times\Delta_\delta\times\Delta_\sigma \quad \forall\alpha\in[\alpha_0,\alpha'']
    \end{equation}    
    Let us define
    \[
        \Omega\coloneqq(\ell_{\alpha_0}\cup\ell_{\alpha''})\cap(\Delta_r^n\times\Delta_\delta)
    \]
    and let $h:\Omega\to\C$ be a holomorphic function defined by
    \[
        \begin{cases}
            h\equiv0 & \text{on $\ell_{\alpha''}$} \\
            h\equiv1 & \text{on $\ell_{\alpha_0}$}
        \end{cases}
    \]
    Since $\Omega$ is a closed analytic subset of a strictly pseudoconvex domain ($\Delta_r^n\times\Delta_\delta$), then there exists a holomorphic extension $\hat{h}:\Delta_r^n\times\Delta_\delta\to\C$ of $h$. Let us define
    \[
        \hat\rho:\Delta_r^n\times\Delta_\delta\to\R, \quad \hat\rho(z,w)=\log|\hat{h}(z,w)|
    \]
    Then $\hat\rho$ is a pluriharmonic function which satisfies
    \begin{equation}
    \label{eq:cases2}
        \begin{cases}
            \hat\rho\equiv-\infty & \text{on $\ell_{\alpha''}$} \\
            \hat\rho\equiv0 & \text{on $\ell_{\alpha_0}$}
        \end{cases}
    \end{equation}
    Now take $\epsilon>0$ and define
    \[
        \psi:\tilde{V}\to\R, \quad \psi(z,w,\zeta_j)=\tilde\rho(z,w,\zeta_j)+\epsilon\hat\rho(z,w)
    \]
    Then $\psi$ is a continuous $(n-1)$-plurisubharmonic function. By construction, we have that
    \begin{equation}
    \label{eq:max_point2}
        \psi(\bar{z},0,\bar\zeta_j)=\rho(\bar{z},\bar\zeta_j)+\epsilon\hat\rho(\bar{z},0)=0
    \end{equation}
    On the other hand, if $\epsilon$ is small enough we obtain that
    \begin{equation}
    \label{eq:neg_points2}
        \psi(z,w,\zeta_j)<0 \quad \forall(z,w,\zeta_j)\in\overline{\Gamma'_j\cap(\Delta_{r',r''}^n\times\Delta_\delta\times\Delta_\sigma)}\cup\Gamma'_{\alpha'',j}
    \end{equation}
    using \eqref{eq:negativity2} and \eqref{eq:cases2}. Let
    \[
        H\coloneqq\bigsqcup_{\alpha\in[\alpha_0,\alpha'')} \ell'_\alpha
    \]
    and define
    \[
        C\coloneqq\Gamma\bigl((g''_j\circ\varpi)_H\bigr)=\bigsqcup_{\alpha\in[\alpha_0,\alpha'')} \Gamma'_{\alpha,j}
    \]
    By~\eqref{eq:localinclusion2}, $C\subset Z_{0,I}^j\cap\tilde{V}=\tilde{Z}_{0,I}^j$. Moreover, as above we have that $(\bar{z},0,\bar\zeta_j)\in C$ and
    \begin{equation}
    \label{eq:boundary2}
        bC=\overline{C}\smallsetminus C\subset\overline{\Gamma'_j\cap(\Delta_{r',r''}^n\times\Delta_\delta\times\Delta_\sigma)}\cup\Gamma'_{\alpha'',j}
    \end{equation}
    Thus
    \[
        \max_{C}\psi\geq\psi(\bar{z},0,\bar\zeta_j)=0>\max_{bC}\psi
    \]
    by~\eqref{eq:max_point2}, \eqref{eq:neg_points2} and \eqref{eq:boundary2}. Hence $\tilde{Z}_{0,I}^j$ is not an $(n-1)$-local maximum set, since $C$ is a relatively compact and relatively open (because $\alpha_0$ is the left endpoint of $I$) subset of $\tilde{Z}^j_{0,I}$ and $\psi$ is an $(n-1)$-plurisubharmonic function defined on a neighbourhood $\tilde{V}\subset B\times i\R\times \C$ of $\overline{C}$. This is a contradiction.

	\bigskip
    Therefore the maps $g_\alpha:\ell_\alpha\to\C^p$ are all holomorphic and $Z_0$ is foliated by the family $\{\Gamma(g_\alpha)\}_{\alpha\in A}$ of $n$-dimensional complex manifolds.
\end{proof}

\section{Some applications}
\label{sec:applications}
In this section, we gather some applications of Theorem~\ref{thm:foliation'}.

First of all, we show that we can obtain a foliation on an $n$-pseudoconcave subset of a $(2n+1)$-dimensional real submanifold of $\C^N$, of some regularity.
\begin{corollary}
\label{coro:submanifold}
    For $1\leq n<N$, let $S\subset\C^N$ be a $(2n+1)$-dimensional real $\mathscr{C}^2$ submanifold and let $Z\subset S$ be an $n$-pseudoconcave subset. Then $Z$ is foliated by $n$-dimensional complex manifolds.
\end{corollary}
\begin{proof}
    Fix $\xi_0\in Z\subset S$. Since $S$ is a $(2n+1)$-dimensional $\mathscr{C}^2$ real submanifold of $\C^N$, then there exists a neighbourhood $U\subset\C^N$ of $\xi_0$ and $\mathscr{C}^2$ functions $r_1,\dots,r_k:U\to\R$ (where $k=2N-2n-1>0$) such that
    \[
        S\cap U=\{r_1=\dots=r_k=0\}
    \]
    Let us define $\phi\coloneqq r_1^2+\dots+r_k^2:U\to\R$. Then $\phi$ is a $\mathscr{C}^2$ function. Note that
    \[
        \frac{\partial^2\phi}{\partial\xi_i\partial\bar\xi_j}=\sum_{l=1}^k2\frac{\partial r_l}{\partial\xi_i}\frac{\partial r_l}{\partial\bar\xi_j}+2r_l\frac{\partial^2 r_l}{\partial\xi_i\partial\bar\xi_j}
    \]
    and then, for all $\xi\in S\cap U$,
    \[
        \frac{\partial^2\phi}{\partial\xi_i\partial\bar\xi_j}(\xi)=2\sum_{l=1}^k\frac{\partial r_l}{\partial\xi_i}(\xi)\frac{\partial r_l}{\partial\bar\xi_j}(\xi)
    \]
    Thus
    \[
    	(\ddc\phi)_\xi=2\sum_{l=1}^k (\d r_l\wedge\dc r_l)_\xi
    \]
    is a positive semidefinite form. In particular, $(\ddc\phi)_\xi$ is null on $H_\xi S$ and positive definite on $H_\xi S^\perp=\C^N/H_\xi S$, where $H_\xi S$ denotes the holomorphic tangent space to $S$ at $\xi$. 
    
    Suppose that $\dim_\C(H_\xi S)<n$. Let $\operatorname{L}_\phi$ be the Levi matrix of $\phi$ and
    \[
        \sigma(\operatorname{L}_\phi(\xi))=\{\alpha_1(\xi),\dots,\alpha_N(\xi)\}
    \]
    be its relative spectrum in $\xi$ (with repetitions allowed). By the assumption,
    \[
        \#\bigl\{j\in\{1,\dots,N\} \mid \alpha_j(\xi)=0\bigr\}\leq n-1
    \]
    Let
    \[
        J_\xi\coloneqq\bigl\{j\in\{1,\dots,N\} \mid \alpha_j(\xi)>0\bigr\}
    \]
    Then $\# J_\xi\geq N-n+1$. Note that the $\alpha_j$'s are continuous functions in $\xi$, since the entries of $\operatorname{L}_\phi$ are continuous. Hence there exists an open neighbourhood $V_\xi\subset\C^N$ of $\xi$ such that $\alpha_j(\xi')>0$ for all $\xi'\in V_\xi$ and for all $j\in J_\xi$. Let
    \[
        V\coloneqq\bigcup_{\xi\in S\cap U} V_\xi
    \]
    Then $V\subset\C^N$ is an open neighbourhood of $S\cap U$. At most shrinking $U$, we can assume that $V=U$. By construction, we have that, for every $\xi'\in U$, $\alpha_j(\xi')>0$ for at least $N-n+1$ indices. Hence $\phi$ is strictly $(n-1)$-plurisubharmonic on $U$. Then we can find a non-negative function $u\in\mathscr{C}^\infty_c(U)$ such that
    \[
        u(\xi_0)>u(\xi) \quad \forall\xi\in(Z\cap U)\smallsetminus\{\xi_0\} 
    \]
    and $\phi+u$ is still $(n-1)$-plurisubharmonic on $U$. Since $\phi|_{S\cap U}\equiv0$, then 
    \[
        (\phi+u)(\xi_0)>(\phi+u)(\xi) \quad \forall\xi\in(Z\cap U)\smallsetminus\{\xi_0\}
    \]
    and thus $Z\cap U$ is not an $(n-1)$-local maximum set. This is a contradiction. Therefore $\dim_\C(H_\xi S)=n$, for all $\xi\in S\cap U$.
    
    In particular, there exists a complex linear coordinate system in $\C^N$ with respect to which
    \[
        T_{\xi_0}S=\C^n\times\R
    \]
    Hence there exists an open ball $B\subset\C^n\times\R$ (centred at the origin), an open neighbourhood $U'\subset\C^N$ of $\xi_0$ and a $\mathscr{C}^2$ function $g:B\to\R\times\C^p$ (where $p=N-n-1\geq0$) such that $S\cap U'=\Gamma(g)$. In particular, $Z_0\coloneqq Z\cap U'$ is the graph of $g|_{X_0}$, for some $X_0\subset B$ closed. Therefore, by Theorem~\ref{thm:foliation'}, $Z$ is foliated by $n$-dimensional complex manifolds.
\end{proof}

Furthermore, we prove a statement analogous to Corollary~\ref{coro:submanifold}, but under weaker hypotheses of regularity on $S$. To achieve this, we will use some of the results obtained in Section~\ref{sec:limit}.

Let $S\subset\C^N$ be a subset such that for each $\xi_0\in S$ there exists a neighbourhood $W\subset S$ of $\xi_0$ and a real linear coordinate system with respect to which $W$ is the graph of a differentiable function $g:B\to\R^q$, with the identification $\C^N=\R^{2N}=\R^{2n+1}_x\times\R^q_y$ (for $1\leq n<N$ and $q=2N-2n-1>0$), where $\pi:\C^N\to\R^{2n+1}$ is the projection and where $B\subset\R^{2n+1}$ is an open ball centred at $x_0=\pi(\xi_0)$. For each $\xi=(x,g(x))\in W$, set
\[
	T_\xi W\coloneqq\Gamma\bigl(Dg(x)\bigr)
\]
where $Dg(x):\R^{2n+1}\to\R^q$ denotes the differential of $g$ at $x$. At most shrinking $W$, assume that
\begin{equation}
\label{eq:perpendicular}
	\#\bigl\{\xi\in W \mid T_\xi W\cap(T_{\xi_0}W)^\perp\neq\{0\}\bigr\}<+\infty
\end{equation}
\begin{remark}
\label{rem:C1_submanifold}
	If $S\subset\C^N$ is a $(2n+1)$-dimensional real $\mathscr{C}^1$ submanifold, then it satisfies all these properties.
\end{remark}
\begin{corollary}
\label{coro:diff_function}
	Let $S\subset\C^N$ be a subset as above. If $Z\subset S$ is an $n$-pseudoconcave subset (for $1\leq n<N$), then it is foliated by $n$-dimensional complex manifolds.
\end{corollary}
\begin{proof}
	Fix $\xi_0\in Z\subset S$. By assumption, there exists an open neighbourhood $W\subset S$ of $\xi_0$ and a differentiable function $g:B\to\R^q$ such that $W=\Gamma(g)$, where $B\subset\R^{2n+1}$ is an open ball centred at $x_0=\pi(\xi_0)$.
	
	By construction, we have that $W=S\cap(B\times\R^q)$ and thus
	\[
		Z_0\coloneqq W\cap Z=(B\times\R^q)\cap Z
	\]
	is relatively open in $Z$. In particular, $Z_0$ is an $(n-1)$-local maximum set.
	
	Consider the set $\Lambda=\{\lambda\in\R \mid \lambda>0\}$ and define on $\Lambda$ the preorder
	\[
		\lambda\preceq\lambda' \iff \lambda'\leq\lambda
	\]
	Then $(\Lambda,\preceq)$ is a directed set. For each $\lambda\in\Lambda$, let us define
	\[
		Z_{0,\lambda}\coloneqq\frac{1}{\lambda}(Z_0-\xi_0)
	\]
	Since $Z_0$ is an $(n-1)$-local maximum set, then so is $Z_{0,\lambda}$. Therefore $\{Z_{0,\lambda}\}_{\lambda\in\Lambda}$ is a net of $(n-1)$-local maximum sets in $\C^N$.
	
	Note that $0\in Z_{0,\lambda}$ for all $\lambda\in\Lambda$ and thus $0\in\limsup_\lambda Z_{0,\lambda}$. In particular, $\{Z_{0,\lambda}\}$ does not escape to infinity, by Remark~\ref{rem:limsup}. Now, for every $\lambda\in\Lambda$, let us define $Y_\lambda\coloneqq\overline{Z_{0,\lambda}}\smallsetminus Z_{0,\lambda}$. We want to prove that $\{Y_\lambda\}$ escapes to infinity. First of all, set $Y\coloneqq\overline{Z_0}\smallsetminus Z_0$ and note that
	\[
		Y_\lambda=\frac{1}{\lambda}(Y-\xi_0)
	\]
	Since $Z_0\subset B\times\R^q$ is closed and contained in $\Gamma(g)$, then $Y\subset bB\times\R^q$. Note that, if $r>0$ is the radius of the ball $B\subset\R^{2n+1}$, then $\norm{\tilde{z}-\xi_0}\geq r$ for all $\tilde{z}\in bB\times\R^q$. Hence
	\[
		\norm{z}=\frac{1}{\lambda}\norm{\tilde{z}-\xi_0}\geq\frac{r}{\lambda} \quad \forall z=\frac{1}{\lambda}(\tilde{z}-\xi_0)\in Y_\lambda
	\]
	and thus
	\[
		\lim_{\lambda\to0^+} \operatorname{dist}(0,Y_\lambda)=+\infty
	\]
	Therefore $\{Y_\lambda\}$ escapes to infinity, by Remark~\ref{rem:escaping}.
	
	Thus the contingent cone to $Z_0$ at $\xi_0$
	\[
		T^+_{\xi_0}Z_0\coloneqq\limsup_{\lambda\in\Lambda}Z_{0,\lambda}
	\]
	is a closed $(n-1)$-local maximum set, by Proposition~\ref{prop:limsup}. On the other hand, we can consider the contingent cone to $W$ at $\xi_0$
	\[
		T^+_{\xi_0}W\coloneqq\limsup_{\lambda\in\Lambda} \frac{1}{\lambda}(W-\xi_0)
	\]
	Since $Z_0\subset W$, then $T_{\xi_0}^+Z_0\subset T_{\xi_0}^+W$. Moreover, $T_{\xi_0}^+W\subset T_{\xi_0}W$, by~\cite[Proposition~2.3]{bigo_greco:characterizations}. Hence $T_{\xi_0}^+Z_0\subset T_{\xi_0}W$ and thus
	\[
		T_{\xi_0}W=\bigcup_{v\in T_{\xi_0}W} (v+T_{\xi_0}^+Z_0)
	\]
	Therefore $T_{\xi_0}W$ has the $(n-1)$-local maximum property. By Lemma~\ref{lemma:type}, we obtain that
	\[
		T_{\xi_0}W=\C^n\times\R
	\]
	In particular, we also obtain that $(T_{\xi_0}W)^\perp=\R\times\C^p$, for $p=N-n-1\geq 0$.
	
	Let $\pi_1:\C^N\to T_{\xi_0}W$ and $\pi_2:\C^N\to (T_{\xi_0}W)^\perp$ be two orthogonal projections, and let $\varpi_1\coloneqq\pi_1|_W$ and $\varpi_2\coloneqq\pi_2|_W$. Moreover, consider the homeomorphism
	\[
		G:B\to W, \quad G(x)=\bigl(x,g(x)\bigr)
	\]
	Note that both $G$ and $G^{-1}$ are differentiable maps, since $g$ is so. Let us define
	\[
		F:B\to T_{\xi_0}W, \quad F\coloneqq\varpi_1\circ G
	\]
	Then $F$ is a differentiable map. Note that, for each $x\in B$, $DF(x)=D\varpi_1(\xi)\circ DG(x)$, where $\xi=G(x)$. Since both $G$ and $G^{-1}$ are differentiable maps, then $DG(x)$ is invertible. On the other hand, we have that $D\varpi_1(\xi)=\pi_1|_{T_\xi W}$. By~\eqref{eq:perpendicular}, we obtain that $D\varpi_1$ has at most finitely many critical points. Hence also $DF$ has at most finitely many critical points. Therefore, by a version of the Inverse Function Theorem for differentiable maps (see~\cite[Theorem~1.3]{li:implicit}), there exists a neighbourhood $U\subset B$ of $x_0$ and a neighbourhood $V\subset T_{\xi_0}W$ of $0$ such that $F|_U:U\to V$ is a homeomorphism. Set $H\coloneqq(F|_U)^{-1}:V\to U$.
	
	Let us define
	\[
		f:V\to(T_{\xi_0}W)^\perp, \quad f\coloneqq\varpi_2\circ G\circ H
	\]
	Then $f$ is continuous. Set $\tilde{W}\coloneqq G(U)$. Then $\tilde{W}$ is an open subset of $W$. Clearly, for each $\xi\in\tilde{W}$ we have that $\xi=(\varpi_1(\xi),\varpi_2(\xi))$ and, moreover, there exists a unique $x\in U$ such that $\xi=G(x)$. Note that
	\[
		\varpi_1(\xi)=\varpi_1\bigl(G(x)\bigr)=F(x)=\nu
	\]
	for a unique $\nu\in V$. On the other hand,
	\[
		\varpi_2(\xi)=\varpi_2\bigl(G(x)\bigr)=\varpi_2\bigl(G(H(\nu))\bigr)=f(\nu)
	\]
	Thus $\xi=(\nu,f(\nu))$ for a unique $\nu\in V$. Therefore $\tilde{W}=\Gamma(f)$.
	
	Hence we have obtained a continuous function $f:V\to\R\times\C^p$, where $V\subset\C^n\times\R$ is an open ball centred at the origin, such that $\tilde{W}=\Gamma(f)$. In particular, $\tilde{Z}_0\coloneqq Z_0\cap\tilde{W}$ is the graph of $f|_{X_0}$, for some $X_0\subset V$ closed. Therefore, by Theorem~\ref{thm:foliation'}, $Z$ is foliated by $n$-dimensional complex manifolds.
\end{proof}

We conclude with an immediate consequence of Corollary~\ref{coro:diff_function}. Before stating it, let us recall the notion of Zariski tangent space in the $\mathscr{C}^k$ category (for $k\geq 1$).
\begin{definition}
	Let $M$ be a $\mathscr{C}^k$ manifold (for $k\geq 1$) and let $Z\subset M$ be an arbitrary subset. Let
	\[
		\mathcal{I}_Z^k\coloneqq\{f\in\mathscr{C}^k(M) \mid f|_Z\equiv 0\}
	\]
	For each $p\in Z$, the set
	\[
		T_p^kZ\coloneqq\{v\in T_pM \mid \d_pf(v)=0 \ \ \forall f\in\mathcal{I}_Z^k\}
	\]
	is called \emph{tangent space to $Z$ at $p$}, of class $\mathscr{C}^k$.
\end{definition}
\begin{remark}
\label{rem:loc_submanifold}
	It is straightforward to prove that the dimension of $T_p^kZ$ is the minimum $n\in\N$ such that $Z$ is locally contained, near $p$, in an embedded $n$-dimensional $\mathscr{C}^k$ submanifold $N$ of $M$ and that, in this case, $T_p^kZ=T_pN$ (see~\cite[Proposition~2.6(c)]{ara_mon:core}).
\end{remark}
By Remark~\ref{rem:loc_submanifold} and Corollary~\ref{coro:diff_function}, we immediately obtain the following (recall also Remark~\ref{rem:C1_submanifold}).
\begin{corollary}
	Let $Z\subset\C^N$ be an $n$-pseudoconcave subset (for $1\leq n<N$) such that
	\[
		\dim_\R(T_p^kZ)=2n+1 \quad \forall p\in Z
	\]
	(for $k\geq 1$). Then $Z$ is foliated by $n$-dimensional complex manifolds.
\end{corollary}

\section{Sharpness of the index of pseudoconcavity}
\label{sec:sharpness}
In this final section, we show that the choice for the index of pseudoconcavity in Theorem~\ref{thm:foliation'} is sharp. Namely, for $m<n$ we construct an $m$-pseudoconcave subset of a smooth graph, defined over $\C^n\times\R$, which has no analytic structure.

To this end, let $f:\C^n\to\C^q$ be a holomorphic function, for $n\geq 2$ and $q\geq 1$, and let $\pi:\C^{n+q}\to\C^n$ be the projection onto the first $n$ coordinates. Then $\tilde\pi\coloneqq\pi|_{\Gamma(f)}$ is a biholomorphism. Now let $E\subset\C^n$ be the closed subset given by~\cite[Theorem~1.1]{har_shch_tom:wermer}. Then $E$ has no analytic structure and $\C^n\smallsetminus E$ is pseudoconvex, i.e., $E$ is an $(n-1)$-pseudoconcave subset of $\C^n$. Let
\[
	Z\coloneqq\Gamma(f|_E)
\]
Then $Z$ is clearly a closed subset of $\C^{n+q}$.
\begin{lemma}
	The subset $Z$ has no analytic structure.
\end{lemma}
\begin{proof}
	By contradiction, suppose that there exists an analytic variety $A\subset Z$, of positive dimension. Let $B\coloneqq\tilde\pi(A)$. Then $B\subset\tilde\pi(Z)=E$ and $B$ is an analytic variety of positive dimension, since $\tilde\pi$ is a biholomorphism. This contradicts the fact that $E$ has no analytic structure.
\end{proof}

Now recall that $E$ is an $(n-1)$-pseudoconcave subset of $\C^n$, i.e., it has the $(n-2)$-local maximum property in $\C^n$.
\begin{lemma}
	The closed subset $Z$ is an $(n-2)$-local maximum set in $\C^{n+q}$.
\end{lemma}
\begin{proof}
	By contradiction, suppose that there exists $p_0\in Z$ such that for each neighbourhood $U\subset\C^{n+q}$ of $p_0$ there exists a compact subset $K\subset U$, with $Z\cap K\neq\emptyset$ compact, and a smooth $(n-2)$-plurisubharmonic function $\phi:U\to\R$ such that
	\[
		\max_{Z\cap K}\phi>\max_{Z\cap bK}\phi
	\]
	where $bK$ denotes the topological boundary of $K$. In particular, there exists $\xi_0\in Z\cap\mathring{K}$ such that
	\[
		\phi(\xi_0)>\phi(\xi) \quad \forall\xi\in Z\cap bK
	\]
	Let $U\subset\C^{n+q}$ be a neighbourhood of $p_0$, with coordinates $(z_1,\dots,z_n,w_1,\dots,w_q)$, such that
	\[
		\Gamma(f)\cap U=\{w_1=\dots=w_q=0\}
	\]
	Take $K,\phi,\xi_0$ relative to this $U$.
	
	Set $V\coloneqq\pi(U)$. Then $V$ is an open neighbourhood of $z_0\coloneqq\pi(\xi_0)$ in $\C^n$. Define
	\[
		\tilde\phi:V\to\R, \quad \tilde\phi(z)=\phi(z,0)
	\]
	(here we are denoting as $(z,w)$ the coordinates on $U$). Then $\tilde\phi$ is a smooth function. We want to prove the $\tilde\phi$ is also $(n-2)$-plurisubharmonic. Note that, for each $z\in V$,
	\[
		\frac{\partial^2\tilde\phi}{\partial z_j\partial\bar{z}_k}(z)=\frac{\partial^2\phi}{\partial z_j\partial\bar{z}_k}(z,0) \quad \forall j,k\in\{1,\dots,n\}
	\]
	and thus $\mathrm{L}_{\tilde\phi}(z)=(\mathrm{L}_\phi(z,0))^n_n$, where $\mathrm{L}_\psi$ denotes the Levi matrix of a function $\psi$ and $A^n_n\in\Mat(n\times n)$ denotes the principal minor of a matrix $A\in\Mat(N\times N)$ given by the first $n$ rows and $n$ columns, i.e.,
	\[
		\mathrm{L}_\phi(z,0)=
		\begin{pmatrix}
			\mathrm{L}_{\tilde\phi}(z) & * \\
			* & *
		\end{pmatrix}
	\]
	For $\xi\in U$, let
	\[
		\sigma\bigl(\mathrm{L}_\phi(\xi)\bigr)=\{\alpha_1(\xi),\dots,\alpha_{n+q}(\xi)\}
	\]
	be the spectrum of $\mathrm{L}_\phi(\xi)$ (with repetitions allowed), ordered decreasingly. Since $\phi$ is $(n-2)$-plurisubharmonic, then $\mathrm{L}_\phi(\xi)$ has at most $n-2$ negative eigenvalues and thus $\alpha_{q+2}(\xi)\geq 0$. Now, for $z\in V$, let
	\[
		\sigma\bigl(\mathrm{L}_{\tilde\phi}(z)\bigr)=\{\beta_1(z),\dots,\beta_n(z)\}
	\]
	be the spectrum of $\mathrm{L}_{\tilde\phi}(z)$ (with repetitions allowed), ordered decreasingly. If
	\[
		A\coloneqq
		\begin{pmatrix}
			\mathrm{I}_n \\
			\vec{0}_{q\times n}
		\end{pmatrix}
		\in\Mat_\C\bigl((n+q)\times n\bigr)
	\]
	then $A^\dagger A=\mathrm{I}_n$ and $A^\dagger\mathrm{L}_\phi(z,0)A=\mathrm{L}_{\tilde\phi}(z)$, where $A^\dagger$ denotes the conjugate transpose of $A$. Thus, by Poincaré Separation Theorem (see~\cite[Theorem~2.1]{rao:separation}),
	\[
		\beta_2(z)\geq\alpha_{q+2}(z,0)\geq 0
	\]
	Therefore $\mathrm{L}_{\tilde\phi}(z)$ has at most $n-2$ negative eigenvalues. By the arbitrariness of $z\in V$, we conclude that $\tilde\phi$ is an $(n-2)$-plurisubharmonic function.
	
	Since $\xi_0\in Z\cap\mathring{K}\subset\Gamma(f)\cap U$, then $\xi_0=(z_0,0)$. Analogously, for each $\xi\in Z\cap bK$ we have that $\xi=(z,0)$ for a unique $z\in\pi(Z)=E$. By assumption, we have that
	\begin{equation}
	\label{eq:contr_locmax}
		\tilde\phi(z_0)=\phi(\xi_0)>\phi(\xi)=\tilde\phi(z) \quad \forall z\in\pi(Z\cap bK)
	\end{equation}
	Take a compact subset $C\subset V$ such that $z_0\in E\cap\mathring{C}$ and $E\cap bC\subset\pi(Z\cap bK)$. 
	By~\eqref{eq:contr_locmax}, we obtain that
	\[
		\max_{E\cap C}\tilde\phi\geq\tilde\phi(z_0)>\max_{E\cap bC}\tilde\phi
	\]
	Therefore $E$ does not have the $(n-2)$-local maximum property. This is a contradiction.
\end{proof}

We have thus constructed, given a holomorphic function $f:\C^n\to\C^q$, an $(n-1)$-pseudoconcave subset $Z$ of $\Gamma(f)$ with no analytic structure. Now consider the identification $\C^{n+q}=\C^n_z\times\C_w\times\C^{q-1}_\zeta$ and let
\[
	\varpi:\C^q\to\R\times\C^{q-1}, \quad \varpi(w,\zeta)=\bigl(\Im(w),\zeta\bigr)
\]
be the projection. Define
\[
	g:\C^n\times\R\to\R\times\C^{q-1}, \quad g(z,\lambda)=\varpi\bigl(f(z)\bigr)
\]
Clearly, $g$ is a smooth function. Writing $f=(f_1,f_2):\C^n_z\to\C_w\times\C^{q-1}_\zeta$, we have that
\[
	g(z,\lambda)=\varpi\bigl(f_1(z),f_2(z)\bigr)=\bigl(\Im(f_1(z)),f_2(z)\bigr)
\]
Hence $\Gamma(f)\subset\Gamma(g)$. On the other hand, let $h:\C^n\to\R$ be a smooth function such that $h|_E\equiv 0$ and $h|_{\C^n\smallsetminus E}>0$. We can take $h$ such that all its derivatives are null on $E$. Define
\[
	\tilde{h}:\C^n\times\R\to\R\times\C^{q-1}, \quad \tilde{h}(z,\lambda)=\bigl(h(z),0\bigr)
\]
and let
\[
	\tilde{g}:\C^n\times\R\to\R\times\C^{q-1}, \quad \tilde{g}=g+\tilde{h}
\]
Then $\tilde{g}$ is a smooth function.
\begin{lemma}
	The graphs $\Gamma(\tilde{g})$ and $\Gamma(f)$ intersect exactly at $Z$.
\end{lemma}
\begin{proof}
	Note that every $\xi\in\Gamma(\tilde{g})$ is of the form $(\eta,\tilde{g}(\eta))$ for a unique $\eta\in\C^n\times\R$. On the other hand, every $\xi\in\Gamma(f)$ is of the form $(z,f_1(z),f_2(z))$ for a unique $z\in\C^n$. Hence $\xi\in\Gamma(\tilde{g})\cap\Gamma(f)$ if and only if $\eta=(z,\Re(f_1(z)))$ and $\tilde{g}(\eta)=(\Im(f_1(z)),f_2(z))$. Note that
	\begin{align*}
		\tilde{g}\bigl(z,\Re(f_1(z))\bigr) &= g\bigl(z,\Re(f_1(z))\bigr)+\tilde{h}\bigl(z,\Re(f_1(z))\bigr) \\
		&= \varpi\bigl(f(z)\bigr)+\bigl(h(z),0\bigr) \\
		&= \bigl(\Im(f_1(z))+h(z),f_2(z)\bigr)
	\end{align*}
	Thus $\xi$ is in the intersection if and only if $h(z)=0$. This holds if and only if $z\in E$, i.e., $\xi\in Z$. Therefore $\Gamma(\tilde{g})\cap\Gamma(f)=Z$.
\end{proof}

Therefore we have constructed a smooth function $\tilde{g}:\C^n\times\R\to\R\times\C^p$ (for $n\geq 2$ and $p\geq 0$) and a closed $(n-1)$-pseudoconcave subset $Z\subset\Gamma(\tilde{g})$ which has no analytic structure. Note that we cannot find a holomorphic function $\tilde{f}:\C^n\to\C^q$ such that $Z\subset\Gamma(\tilde{f})\subset\Gamma(\tilde{g})$, as it happens, for example, with $\Gamma(g)$.

We have thus proved that the choice for the index of pseudoconcavity in Theorem~\ref{thm:foliation'} is sharp. Namely, let $Z\subset\C^N$ be an $m$-pseudoconcave subset which is locally the graph of a continuous function over a closed subset of $\C^n\times\R$. Recall the following fact.
\begin{remark}
	Since $Z$ is $m$-pseudoconcave, then it is also $k$-pseudoconcave, for every $k\leq m$.
\end{remark}
Therefore we have two cases:
\begin{itemize}
	\item if $m\geq n$, then $Z$ is $n$-pseudoconcave and thus it is foliated, by Theorem~\ref{thm:foliation'};
	\item if $m<n$, then $Z$ can have no analytic structure, as shown above.
\end{itemize}

\printbibliography

\end{document}